\documentclass[a4paper,10pt]{amsart}
\usepackage{amsmath,amsfonts,amssymb,amsthm, amsfonts, amscd, url}
\usepackage[latin1]{inputenc}
\usepackage[english]{babel}
\usepackage{url}
\usepackage{setspace}

\usepackage[cmtip,arrow]{xy}
\usepackage{pb-diagram,pb-xy}

\swapnumbers 
\theoremstyle{plain}
\newtheorem{thm}{Theorem}[section]
\newtheorem{prop}[thm]{Proposition}
\newtheorem{lemma}[thm]{Lemma}

\theoremstyle{remark}
\theoremstyle{definition}
\newtheorem{rem}[thm]{Remark}
\newtheorem{rems}[thm]{Remarks}

\newtheorem{remsdefs}[thm]{Remarks-Definitions}

\newtheorem{defi}[thm]{Definition}
\newtheorem{defis}[thm]{Definitions}

\newtheorem{notas}[thm]{Notations}

\usepackage[english]{babel}

\title[Real logarithms of semi-simple matrices]{Real logarithms of semi-simple matrices}

\author{Donato Pertici}

\begin{document}

\parindent 0pt
\selectlanguage{english}

\maketitle

\vspace*{-0.2in}

\begin{center}
{\scriptsize Dipartimento di Matematica  e Informatica, Viale Morgagni 67/a, 50134 Firenze, ITALIA

\vspace*{0.02in}

donato.pertici@unifi.it,  \   http://orcid.org/0000-0003-4667-9568}

\end{center}


\vspace*{-0.1in}

\begin{abstract}
We study the differential structure of the set of real logarithms of a non-singular real matrix, under the assumption  that the matrix is either semi-simple or orthogonal.
\end{abstract}


{\small \tableofcontents}

\renewcommand{\thefootnote}{\fnsymbol{footnote}}
\footnotetext{
This research was partially supported by GNSAGA-INdAM (Italy).
}
\renewcommand{\thefootnote}{\arabic{footnote}}
\setcounter{footnote}{0}

\vspace*{-0.3in}


\smallskip

{\small Mathematics~Subject~Classification~(2020): 15A24, 53C30, 15B10.}

\bigskip

\section*{Introduction}\label{intro}

The aim of this work is to study the differential geometric properties of the sets of real logarithms of real semi-simple matrices and of skew-symmetric real logarithms of special orthogonal matrices. As far as we know, such studies have never been done before. More generally, there are not many works that study the real logarithms of a matrix from a theoretical point of view. The most known is an old work by Culver (\cite{Culv1966}), in which, among other things, the author proves that a non-singular real square matrix $M$ has a real logarithm if and only if each of its Jordan blocks corresponding to a negative eigenvalue occurs an even number of times. Furthermore, Culver provides necessary and sufficient conditions, in terms of the Jordan blocks of M, for the real logarithm of M to be unique and for the set of real logarithms of M to be countable. A simpler exposition of some of these findings can be found in \cite{Nun1989}.

After a first preliminary Section in which, in particular, the main homogeneous spaces involved in the structure theorems of the sets of real logarithms and of skew-symmetric real logarithms are defined, in Section \ref{Top} \ we determine, in the simplest cases, some of the homotopy groups of these homogeneous spaces. 

The main result of Section \ref{LOG} is Theorem \ref{logaritmi}, which states that the set of real logarithms of a given semi-simple matrix is a countable disjoint union of simply connected differentiable submanifolds of $M(n,\mathbb{R})$, all diffeomorphic to suitable homogeneous spaces and whose dimensions depend on the eigenvalues of the matrices constituting each of them. This Theorem also states that the second homotopy group of each of these components is a free abelian group, whose rank can be expressed as a function of the eigenvalues of the matrices constituting the given component.

In Section \ref{Skew-log} we prove a similar Theorem for the set of skew-symmetric real logarithms of a given special orthogonal matrix (see Theorem \ref{manifold-O-con-indici}). In this case, each connected component of this set is a compact submanifold of $\mathfrak{so}(n)$.

At the end of Sections \ref{LOG} and \ref{Skew-log}, respectively, the sets of generalized principal real logarithms and of principal skew-symmetric real logarithms for suitable matrices are studied. The most interesting results we have obtained are stated in Theorems \ref{logar-princ} and \ref{antis-log-princ}. Note that our Definitions \ref{prin-gen} and \ref{prin-antis-log} are more general than the usual definition of  principal logarithm (see \cite[Thm.\,1.31 p.\,20]{Hi2008}). In fact, our principal logarithms are also defined for matrices with negative eigenvalues, even if, in general, they are not unique. For a study of the set of principal skew-symmetric real logarithms of a given special orthogonal matrix see also \cite[\S 3]{DoPe2018}.

We remark that the methods we use in this paper are very similar to those used in \cite{DoPe2020} to study the set of real square roots of suitable real matrices.

\section{Preliminary facts}\label{prelim}

\begin{notas}\label{notazioni}
a) \ In this paper, for any integer $n \geq 1$, we denote

-\ $M(n,\mathbb{R})$: the $\mathbb{R}$-vector space of the real square matrices of order $n$;

-\ $GL(n,\mathbb{R})$ (and $GL^+ (n,\mathbb{R})$): the multiplicative group of the non-singular real matrices of order $n$ (with positive determinant);

-\ $\mathcal{O}(n)$ (and $S\mathcal{O}(n)$): the multiplicative group of the real orthogonal matrices of order $n$ (with determinant $1$);

-\ $\mathfrak{so}(n)$: the Lie algebra of the skew-symmetric real matrices of order $n$;

-\ $M(n,\mathbb{C})$: the $\mathbb{C}$-vector space of the complex square matrices of order $n$;

-\ $GL(n,\mathbb{C})$: the multiplicative group of the non-singular complex matrices of order $n$;

-\ $U(n)$: the multiplicative group of the complex unitary matrices of order $n$; 

-\ $I_{_n}$: the identity matrix of order $n$;

-\ $\mathrm{O}_{_n}$: the null matrix of order n;

-\ ${\bf i}$: the imaginary unit.

\smallskip

We write $\bigsqcup_j X_j$ to emphasize the union of mutually disjoint sets $X_j$; furthermore we denote by $|S|$ the cardinality of any given finite set $S$ \ and by \ $\delta_{(i,j)}$ \ the usual \emph{Kronecker delta} defined by $\delta_{(i,j)}=1$\ \ if \ $i = j$, \ and \ $0$ \ otherwise.
\smallskip

b) \ For every $A \in M(n,\mathbb{C})$, $tr(A)$ is its trace, $A^T$ is its transpose, $A^*:=\overline{A\,}^T$ is its transpose conjugate, $det(A)$ is its determinant and, provided that $det(A) \ne 0$, $A^{-1}$ is its inverse;
furthermore $exp(A)=\sum\limits_{i=0}^\infty \dfrac{A^i}{i!}$ \ denotes the exponential of $A$.

If $C \in GL(n,\mathbb{R})$ , we denote by $\Lambda_{_C}: X \mapsto  CXC^{-1}$ the inner automorphism of $GL(n,\mathbb{R})$ associated to the matrix $C$.

For every $\theta \in \mathbb{R}$, we denote
 $E_{\theta}:=\begin{pmatrix} 
 \cos (\theta) & - \sin (\theta) \\ 
\sin (\theta) &  \cos (\theta)
\end{pmatrix}$
 and $E:= E_{\pi /2}= 
 \begin{pmatrix} 
 0 & -1 \\ 
1 &  0
\end{pmatrix}$; 
hence $E_\theta = \cos (\theta)I_2 + \sin (\theta) E$.

It is easy to check that \ $exp(\delta E)= E_{\delta}$ , for every $\delta \in \mathbb{R}$, and from this we get

$  (\bigtriangleup)  \ \ \ \ \ \ \ \ \ \ \ \ \ \ \ \ \ \ \ \ \ \ \ \  exp(\alpha I_2 +\delta E)= e^{\alpha} E_{\delta}$, \ for every $\alpha, \delta \in \mathbb{R}$ .

If $B_1, \cdots , B_m$ are square matrices (of various orders), $B_1 \oplus \cdots \oplus B_m$ is the block diagonal square matrix with $B_1, \cdots , B_m$ on its diagonal and, for every square matrix $B$, \ \ $B^{\oplus m}$ denotes $B \oplus \dots \oplus B$ ($m$ times). 

It is clear that $exp(B_1 \oplus \cdots \oplus B_m)= exp(B_1) \oplus \cdots \oplus exp(B_m)$ , for every $ B_1, \cdots , B_m$.

If $\mathcal{S}_1, \dots , \mathcal{S}_m$ are sets of square matrices, then  $\mathcal{S}_1 \oplus \dots \oplus \mathcal{S}_m$ denotes the set of all matrices $B_1 \oplus \cdots \oplus B_m$ with $B_j \in \mathcal{S}_j$\ , for every $j$.

To give a full generality to the results of this paper (and to their proofs), it is necessary to establish agreements on the notations that we will use: if \ $s$ is a non-negative integer parameter, whenever, in any formula, we write any term as $\sum\limits_{i=1}^s  \ (\cdots), \ \ \bigoplus\limits_{i=1}^s \ (\cdots)$ \ or \ $\prod\limits_{i=1}^s \ (\cdots)$, \ we mean that, if $s=0$, this sum, this direct sum or this product must not appear in the related formula.
Similar considerations for \ $\sum\limits_{i \in I} \ (\cdots)$\ and\ $\bigoplus\limits_{i \in I} \ (\cdots)$, whenever the set $I$ is empty. 
We also assign a meaning to the zero-order matrices $I_{_0}$,  $\mathrm{O}_{_0}$ and to the zero-order groups $S\mathcal{O}(0), \ \mathcal{O}(0), \ GL(0,\mathbb{R}), \ GL^+(0,\mathbb{R})$, defining them all equal to a single (phantom) point \ $\mathcal{Q}$ \ which, conventionally, satisfies the following conditions:

$\lambda \mathcal{Q} = \mathcal{Q}$, \ for every $\lambda \in \mathbb{C}$; \ \ \ \ \ \ \ $\mathcal{Q} \oplus B = B \oplus \mathcal{Q} = B$, \ for any complex square matrix $B$; \ \ \ \ \ \ \ $\mathcal{Q} \oplus \mathcal{S} = \mathcal{S} \oplus \mathcal{Q} = \mathcal{S}$, \ for any set of complex square matrices $\mathcal{S}$.

Moreover, we also agree that the zero-multiplicity eigenvalues of a given matrix $X \in M(n,\mathbb{C})$ are all complex numbers that are not eigenvalues of $X$, while when we say that $G$ is a free abelian group of rank zero, this means that $G = \lbrace0\rbrace$.

For all other notations and information on matrices, not explicitly recalled here, we refer to \cite{HoJ2013} and to \cite{Hi2008}.
\end{notas}

\begin{remsdefs}\label{rho}
a) The mapping 
$\rho_1: \mathbb{C} \to M(2,\mathbb{R})$ defined by

$\rho_1(z) =  Re(z) I_2 + Im(z) E$, is a monomorphism of $\mathbb{R}$-algebras such that 
$\rho_1({\overline{z}}) = \rho_1(z)^T$
and $\rho_1(z) \in GL(2,\mathbb{R})$ as soon as $z \ne 0$. 

More generally, for any $h \ge 1$, we define the \emph{decomplexification mapping} 

$\rho_h: M(h,\mathbb{C}) \to M(2h,\mathbb{R})$, which maps the $h \times h$ complex matrix $Z=(z_{ij})$ to the $(2h) \times (2h)$ block real matrix  $(\rho_1(z_{ij}))$, having $h^2$ blocks of order $2 \times 2$. 

We have 
$tr(\rho_h(Z)) = 2 Re(tr(Z)))$, $det(\rho_h(Z)) = |det(Z)|^2$ and, moreover, $\rho_h$ is a monomorphism of $\mathbb{R}$-algebras, whose restriction to $GL(h,\mathbb{C})$ is a monomorphism of Lie groups from $GL(h,\mathbb{C})$ into $GL(2h,\mathbb{R})$.

Furthermore we have: $\rho_h(Z^*) = \rho_h(Z)^T$, so the restriction of the monomorphism $\rho_h$ to $U(h)$ maps $U(h)$ into $S\mathcal{O}(2h)$ and $\rho_h (U(h))=  \rho_h (GL(h,\mathbb{C})) \cap S\mathcal{O}(2h) = \rho_h (GL(h,\mathbb{C})) \cap \mathcal{O}(2h)$. 
To simplify the notations and in absence of ambiguity, from now on we omit to write the symbol $\rho_h$, so, for instance, we can consider $M(h,\mathbb{C})$ as an $\mathbb{R}$-subalgebra of $M(2h,\mathbb{R})$, $GL(h,\mathbb{C})$ as a Lie subgroup of $GL(2h,\mathbb{R})$ and  $U(h)$ as a Lie subgroup of $S\mathcal{O}(2h)$.

b) For every matrix $B \in M(n,\mathbb{R})$ we denote
$\mathcal{C}_B := \{X \in GL(n,\mathbb{R}) : BX=XB\}$.
It is easy to prove that $\mathcal{C}_B$ is a closed subgroup of $GL(n,\mathbb{R})$.

If $X \in M(n,\mathbb{R})$, we also denote by $\sigma(X)$ the set of its distinct complex eigenvalues.
\end{remsdefs}
\begin{lemma}\label{commutazione}
Let $D=\bigoplus\limits_{j=1}^t D_j  \in M(n,\mathbb{R})$, where each $D_j$ is a semi-simple real square matrix of order $n_{_j}$ (with $\sum\limits_{j=1}^t n_{_j}=n$) and assume that $\sigma(D_j)\cap\sigma(D_h)=\emptyset$, as soon as $j\neq h$. Then we have \ \ 
$\mathcal{C}_D = \bigoplus\limits_{j=1}^t \mathcal{C}_{D_j} .$
\end{lemma}

\begin{proof}
Any matrix $A \in GL(n, \mathbb{R})$ can be written in blocks as $A= (A_{ij})$, where $A_{ij}$ is an $n_i \times n_j$ real matrix, for $i,j=1, \cdots, t$. The condition $A D =D A$ is equivalent to $A_{ij} D_j = D_i A_{ij}$, for every $1\leq i,j \leq t$. 

Fix $i, j \in \lbrace1, \cdots, t\rbrace$ with $i\neq j$. Since $D_j$ is semi-simple, there is a basis of $\mathbb{C}^{n_j}$, say $\lbrace v_{_1}, \cdots, v_{n_j}\rbrace$, consisting of eigenvectors of $D_j$; hence, if $v$ is an element of this basis, then $D_j(v)=\lambda v$\ , for some $\lambda \in \sigma(D_j)$, and so \ $ D_i A_{ij} (v) = A_{ij} D_j (v) = \lambda A_{ij}(v)\ .$ From the assumptions $\lambda \notin \sigma(D_i)$, hence we conclude that $A_{ij}(v)=0$, for every $v \in \lbrace v_{_1}, \cdots, v_{n_j}\rbrace$, and so $A_{ij}=0$. 
Therefore $A= \bigoplus\limits_{j=1}^t A_{jj}$\ , where $A_{jj} \in GL(n_{_j}, \mathbb{R})$ and $A_{jj} D_j = D_j A_{jj}$\ , for $j=1, \cdots, t$ . This concludes the proof.
\end{proof}

Using elementary arguments similar to those of Lemma \ref{commutazione}, it is easy to prove the following:

\begin{lemma}\label{Ethetacommutaz}
Denote $F=E_{\theta}^{\oplus m} \in GL(2m, \mathbb{R})$, where $0 < \theta < \pi$ and $m \geq 1$. Then $\mathcal{C}_F = GL(m, \mathbb{C}) .$
\end{lemma}

\begin{remsdefs}\label{def-spaz-omog}
a) If $G$ is a real Lie group and $H$ is a closed subgroup of $G$, we say that $H$ is a Lie subgroup of $G$. Under this assumption, it is well known that $H$ is also a real Lie group and that the set of left cosets $G/H$ has a unique structure of differentiable manifold (of dimension equal to $\dim_{\mathbb{R}} G - \dim_{\mathbb{R}} H$) such that the natural projection $G \to G/H$ is differentiable and the natural action of $G$ on $G/H$ is a $C^{\infty}$ transitive action. Moreover, $G$ is a principal fiber bundle over $G/H$ with group $H$.

b) In the next Sections we will have to deal with the following homogeneous spaces:
\smallskip

\ \ \ $\widehat{\Theta}_{(\nu_1,\cdots,\nu_s)}=\dfrac{GL(\nu,\mathbb{C})}{\big( \bigoplus\limits_{j=1}^s GL(\nu_j,\mathbb{C})\big)}$\ ,
\ \ \ \ \ \ \ \ \ \ \ $\Theta_{(\nu_1,\cdots,\nu_s)}=\dfrac{U(\nu)}{\big(\bigoplus\limits_{j=1}^s U(\nu_j)\big)}$\ ,
\smallskip

\ \ \ $\widehat{\Gamma}_{(\zeta; \nu_1, \cdots,\nu_s)}=$
$\left\lbrace 
\begin{array}{cc}
\dfrac{GL^+(\zeta+2\nu,\mathbb{R})}{GL^+(\zeta,\mathbb{R})\oplus \big(\bigoplus\limits_{j=1}^{s}GL(\nu_j,\mathbb{C})\big)} \ \ \ \ \ \ \ \ if \ \ \ \ \ \zeta \geq 1 \\
\vphantom{}\\
  \ \ \ \dfrac{GL^+(2\nu,\mathbb{R})}{\big(\bigoplus\limits_{j=1}^s GL(\nu_j,\mathbb{C})\big)} \ \ \ \ \ \ \ \ \ \ \ \ \ \ \ \ \ \ \ if \ \ \ \ \ \zeta=0
 \end{array}
\right. $,
\medskip

\ \ \ $\Gamma_{(\zeta; \nu_1, \cdots,\nu_s) }=\left\lbrace 
\begin{array}{cc}
\dfrac{S\mathcal{O}(\zeta+2\nu)}{S\mathcal{O}(\zeta)\oplus \big(\bigoplus\limits_{j=1}^{s}U(\nu_j)\big)} \ \ \ \ \ \ \ if \ \ \ \ \ \zeta \geq 1 \\
\vphantom{}\\
  \ \ \ \dfrac{S\mathcal{O}(2\nu)}{\big(\bigoplus\limits_{j=1}^s U(\nu_j)\big)} \ \ \ \ \ \ \ \ \ \ \ \ \ if \ \ \ \ \ \zeta=0
 \end{array}
\right. $,
\smallskip
\smallskip

where $\zeta, \nu_1, \cdots, \nu_s$ are integers such that $\zeta \geq 0$,   $\nu_1, \cdots, \nu_s \geq 1 \ (s \geq 1)$ and 

$\nu=\sum\limits_{j=1}^s\nu_j$ .
It is useful to define the spaces $\widehat{\Gamma}_{(\zeta; \nu_1, \cdots,\nu_s)}$ and $\Gamma_{(\zeta; \nu_1, \cdots,\nu_s) }$ \ even when $s=0$ (i.e. when the multi-index $(\zeta; \nu_1, \cdots , \nu_s)$ reduces to $(\zeta)$) and $\zeta \geq 0$, 
by setting them, in all these cases, equal to a single point. Consequently, note that $\Gamma_{(\zeta; \nu_1, \cdots,\nu_s)}$ reduces to a single point if and only if either $s=0, \ \ \zeta \geq 0$ \ or \ $s=1$, $\nu_1= \nu= 1,\ \zeta=0$, while this holds for the space $\widehat{\Gamma}_{(\zeta; \nu_1, \cdots,\nu_s)}$ only for $s=0,\ \ \zeta \geq 0$. Also note that both spaces $\Theta_{(\nu_1,\cdots,\nu_s)}$ and  $\widehat{\Theta}_{(\nu_1,\cdots,\nu_s)}$ are single points if and only if $s=1$ (for every $\nu =\nu_1 \geq 1)$.

c) All the spaces we have defined are connected differentiable manifolds; moreover $\Gamma_{(\zeta; \nu_1,\cdots, \nu_s)}$ and $\Theta_{(\nu_1,\cdots,\nu_s)}$ are also compact. Their  dimensions are the following:\\
$\dim_{\mathbb{R}}\widehat{\Gamma}_{(\zeta; \nu_1,\cdots, \nu_s)}= 4\nu(\nu+\zeta)-2\sum\limits_{j=1}^s \nu_j^2 $,\ \ \ \ 
$\dim_{\mathbb{R}}\widehat{\Theta}_{(\nu_1, \cdots, \nu_s) }=2\nu^2-2\sum\limits_{j=1}^s \nu_j^2 $,\\
$\dim_{\mathbb{R}}\Gamma_{(\zeta; \nu_1,\cdots, \nu_s)}= \nu(2\nu+2\zeta-1)-\sum\limits_{j=1}^s \nu_j^2 $,\ \ \ \ 
$\dim_{\mathbb{R}}\Theta_{(\nu_1, \cdots, \nu_s) }=\nu^2-\sum\limits_{j=1}^s \nu_j^2 $.
We also observe that $\Gamma_{(0, \nu)}, \widehat{\Gamma}_{(0, \nu)}, \Theta_{(\nu_{_1}, \nu_{_2})}, \widehat{\Theta}_{(\nu_{_1}, \nu_{_2})}$ are symmetric spaces, for every $\nu, \nu_{_1}, \nu_{_2} \geq 1.$ Moreover, it can be easily seen that $\Gamma_{(0, 2)}$ is diffeomorphic to the $2$-dimensional sphere.

d) Note that, if $\zeta \geq 1$, the homogeneous spaces $\dfrac{GL(\zeta+2\nu,\mathbb{R})}{GL(\zeta,\mathbb{R})\oplus \big(\bigoplus\limits_{j=1}^{s}GL(\nu_j,\mathbb{C})\big)}$ and 
$\dfrac{\mathcal{O}(\zeta+2\nu)}{\mathcal{O}(\zeta)\oplus \big(\bigoplus\limits_{j=1}^s U(\nu_{j})\big)}$ are diffeomorphic to $\widehat{\Gamma}_{(\zeta; \nu_1, \cdots,\nu_s)}$ and  $\Gamma_{(\zeta; \nu_1, \cdots,\nu_s)}$, respectively (and so they are connected), while, for $\zeta =0$, the spaces $\dfrac{GL(2\nu,\mathbb{R})}{\big(\bigoplus\limits_{j=1}^{s}GL(\nu_j,\mathbb{C})\big)}$ and $\dfrac{\mathcal{O}(2\nu)}{\big(\bigoplus\limits_{j=1}^s U(\nu_{j})\big)}$ have two connected components both diffeomorphic to $\widehat{\Gamma}_{(0; \nu_1, \cdots,\nu_s)}$ and  $\Gamma_{(0; \nu_1, \cdots,\nu_s)}$, respectively;
hence we can say that 
$\dfrac{GL(\zeta+2\nu,\mathbb{R})}{GL(\zeta,\mathbb{R})\oplus \big(\bigoplus\limits_{j=1}^{s}GL(\nu_j,\mathbb{C})\big)}$ and 
$\dfrac{\mathcal{O}(\zeta+2\nu)}{\mathcal{O}(\zeta)\oplus \big(\bigoplus\limits_{j=1}^s U(\nu_{j})\big)}$ \ have \ $2^{\delta_{_{(\zeta,0)}}}$ connected components (if $\zeta +s \geq 1$).
 
\end{remsdefs}

\begin{defi}

Let $M \in M(n,\mathbb{R})$ any matrix. 
We call \emph{real logarithm of $M$} every matrix $X$ of $M(n,\mathbb{R})$, solving the matrix equation $exp(X)=M$ .
\end{defi}

\begin{rem}\label{esist-rad-quadr}
It is well known that $exp(X) \in GL^+(n,\mathbb{R}) $, for every $X \in M(n,\mathbb{R})$ ; hence no matrix with non-positive determinant has real logarithms. Moreover, also the following fact is known (see for instance \cite[Thm.\,1]{Culv1966} or \cite[Thm.\,1.23]{Hi2008}): 

$M \in GL^+(n,\mathbb{R})$ has at least one real logarithm if and only if it has an even number of Jordan blocks of each size, for every negative eigenvalue. 

Note that, if the matrix $M \in GL^+(n,\mathbb{R})$ is semi-simple, then it has at least one real logarithm if and only if each of its negative eigenvalues has even multiplicity.
\end{rem}

\begin{notas}\label{def-XSR}
Assume that $M \in GL^+(n,\mathbb{R})$ is semi-simple and that its (possible) negative eigenvalues have even multiplicity. 

We want to study the following sets:

\smallskip

$\mathcal{L}og(M)$, the set of all real logarithms of $M$ (see \S \ref{LOG});

\smallskip

$\mathcal{L}og_{{\mathfrak{so}(n)}}(M) :=\mathcal{L}og(M) \cap \mathfrak{so}(n)$, the set of all skew-symmetric real logarithms of $M$, when $M$ is supposed to be an element of $S\mathcal{O}(n)$ (see \S \ref{Skew-log}).
\end{notas}

\begin{rem}\label{Popov}
Let $G$ be a real Lie group acting smoothly on a differentiable manifold $X$. The orbit of every $x \in X$ is an immersed submanifold of $X$, diffeomorphic to the homogeneous space $\dfrac{G}{G_x}$, where $G_x$ is the isotropy subgroup of $G$ at $x$. 

This submanifold is not necessarily embedded in $X$, but, if $G$ is compact, then all orbits are embedded submanifolds (see \cite{Orbit}). 
\end{rem}

\section{Some remarks on the homotopy groups of homogeneous spaces}\label{Top}

As we will see, the homogeneous spaces we have defined in Remarks-Definitions \ref{def-spaz-omog} are involved in the study of the real logarithms of an arbitrary matrix. For this reason, in this section we will study some of their topological properties. We begin with some general properties concerning homogeneous spaces.

\begin{rem}\label{Omotopia}
Let $G$ be a connected Lie group with identity $e$ and let $H$ be any connected Lie subgroup of $G$. Denoted by $G/H$ the related homogeneous space and by $\lbrace e \rbrace= H$ the equivalence class of $e$ in the quotient $G/H$, it is known that we have the following exact homotopy sequence, induced by the fibration on the quotient (see \cite [p.90]{Steen1951}):

$\cdots \overset{\delta}{\longrightarrow} \pi_i (H) \overset{\psi}{\longrightarrow} \pi_i (G) \overset {\xi}{\longrightarrow}\pi_i (G/H) \overset{\delta}{\longrightarrow} \cdots \overset {\xi}{\longrightarrow} \pi_2 (G/H) \overset{\delta}{\longrightarrow} \pi_1 (H) \overset{\psi}{\longrightarrow}\pi_1 (G) \overset {\xi}{\longrightarrow} \pi_1 (G/H) \to 0$  .

In this sequence the homotopy groups are based at the point $e$ for $G$ and $H$ and at the point $\lbrace e \rbrace$ for $G/H$ ; the mappings $\psi$ and $\xi$ are, respectively, the homomorphisms induced by the natural inclusion $H \to G$ and by the projection on the quotient $G \to G/H$, while the mappings $\delta$ are the connecting homomorphisms.
\end{rem}

\begin{lemma}\label{fivelemma}
Let $G'$, $H$ and $H'$ be connected Lie subgroups of a connected Lie group $G$, such that $H' \subset G'\cap H$. Suppose $G'$ is a deformation retract of $G$ and $H'$ is a deformation retract of $H$ . Then $\pi_i(G/H)\cong \pi_i(G'/H')$, for every $i \geq 1$.

\end{lemma}

\begin{proof}
From the assumptions, it follows that the natural inclusion: $G' \to G$ is a bundle morphism, i.e. there exist a (natural) inclusion map: $G'/H' \to G/H$ such that the diagram 
$\begin{CD}
G' @>>> G\\
@VVV @VVV\\
G'/H' @>>> G/H
\end{CD}$
commutes. Then, for every $i \geq 2$, we get the following commutative diagram, where the rows are exact sequences (see \cite [p.91]{Steen1951}):

$\begin{CD}
\cdots \pi_i (H') @>\psi'>>\pi_i (G')@>\xi'>>\pi_i(G'/H')@>\delta'>>\pi_{i-1} (H') @>\psi'>>\pi_{i-1} (G') \cdots\\
@VV{f_i}V @VV{j_i}V @VV{l_i}V @VV{f_{i-1}}V @VV{j_{i-1}}V\\
\cdots \pi_i (H) @>\psi>>\pi_i (G)@>\xi>>\pi_i(G/H)@>\delta>>\pi_{i-1} (H) @>\psi>>\pi_{i-1} (G) \cdots
\end{CD}$

\smallskip

\smallskip

Here the maps $f_i$, $j_i$ and $l_i$ are the homomorphisms induced by the natural inclusions. Since all groups are connected, if we define, as usual, $\pi_0(G)= \pi_0(G')= \pi_0(H)=\pi_0(H')$ $ = \lbrace0\rbrace$, the previous commutative diagram also remains valid for $i= 1$. 
Furthermore, since $G'$ and $H'$ are deformation retracts of $G$ and $H$, respectively, all the mappings $f_r, j_r$ are isomorphisms, so that, by the classical Five-Lemma (see, for instance, \cite [p. 129]{Hat2001}, in which the proof also works for non-abelian groups), the mappings $l_i$ are isomorphisms too , for every $i \geq 1$. This concludes the proof of the Lemma.

\end{proof}

\begin{prop}\label{omotopie-isomorfe}
Let $\zeta, \nu_1, \cdots, \nu_s$ be integers such that $\zeta \geq 0$ and $\nu_j \geq 1$, for $j=1, \cdots, s \ \ (s \geq 1)$. Then, for every $i \geq 1$,

$\pi_i(\widehat{\Gamma}_{(\zeta; \nu_1, \cdots, \nu_s)})\cong \pi_i(\Gamma_{(\zeta; \nu_1, \cdots, \nu_s)})$ \ \ \  and \ \ \ 
$\pi_i(\widehat{\Theta}_{(\nu_1, \cdots, \nu_s)})\cong \pi_i(\Theta_{(\nu_1, \cdots, \nu_s)})$. 

\end{prop}

\begin{proof}
At first, we prove that $S\mathcal{O}(n)$ is a deformation retract of $GL^+(n,\mathbb{R})\ \ (n\geq 1)$. \\If $X \in GL^+(n,\mathbb{R})$, by polar decomposition (see \cite[Thm.\,7.3.1  p.\,449]{HoJ2013}), we can write $X= (X X^T)^{1/2} \big((X X^T)^{-1/2}\cdot X\big)$, where $(X X^T)^{1/2}$ is a positive definite symmetric real matrix of order $n$ and $\big((X X^T)^{-1/2}\cdot X \big) \in S\mathcal{O}(n)$. Denoted by $log((X X^T)^{1/2})$ the unique real symmetric logarithm of the positive definite matrix $(X X^T)^{1/2}$, by $j: S\mathcal{O}(n) \to  GL^+(n,\mathbb{R})$ the natural inclusion, 
by $\widehat{r}: GL^+(n,\mathbb{R}) \to S\mathcal{O}(n)$ the retraction such that $X \mapsto \big((X X^T)^{-1/2}\cdot X\big)$, we can define \\$H(X,t)= \exp (t\cdot log((X X^T)^{1/2})) \big((X X^T)^{-1/2}\cdot X\big)$, for every $X \in GL^+(n,\mathbb{R})$ and $t \in [0,1]$. $H$ is a $C^\infty$ homotopy between $j \circ \widehat{r}$ and the identity map of $GL^+(n,\mathbb{R})$, so that $S\mathcal{O}(n)$ is a deformation retract of $GL^+(n,\mathbb{R})$.\\
Likewise, it is possible to prove that $U(n)$ is a deformation retract of $GL(n,\mathbb{C})$, $\big(\bigoplus\limits_{j=1}^s U(\nu_j)\big)$ is a deformation retract of $\big(\bigoplus\limits_{j=1}^s GL(\nu_j,\mathbb{C})\big)$, and $S\mathcal{O}(\zeta)\oplus \big(\bigoplus\limits_{j=1}^{s}U(\nu_j)\big)$ is a deformation retract of
$GL^+(\zeta,\mathbb{R})\oplus \big(\bigoplus\limits_{j=1}^{s}GL(\nu_j,\mathbb{C})\big)$.
Hence the Proposition follows from Lemma \ref{fivelemma}.
\end{proof}

\begin{rem}\label{Quozienti}

Remembering Remarks-Definitions \ref{def-spaz-omog} (b), the spaces $\Gamma_{(0; 1) }$, $\Gamma_{(\zeta) }$ and $\Theta_{(\nu)}$ reduce to a single point and so their homotopy groups are trivial.  

We also recall that the so-called \emph{stable} homotopy groups of the symmetric spaces 
$\Gamma_{(0,\nu)}=\dfrac{S\mathcal{O}(2\nu)}{U(\nu)}$ have been computed by R. Bott in his fundamental work \cite{Bott1959}, while results about \emph{unstable} homotopy groups of $\Gamma_{(0,\nu)}$ have been obtained by various other authors (see, for instance, \cite{Harris1963}, \cite{Oshima1984} and \cite{Mukai1990}).
Among the known results, we will use the following:
\end{rem}

\begin{prop}\label{Bott}
The manifold $\Gamma_{(0,\nu)}$ is simply connected and $\pi_2(\Gamma_{(0,\nu)})\cong \mathbb{Z}\ ,$ for every $\nu \geq 2$.
\end{prop}
We will study the other cases in the next Propositions of this Section.

\begin{prop}\label{Gamma-omotopia}
Let  $\zeta, \nu_1, \cdots, \nu_s$ be integers such that $\zeta \geq 0$, $\nu_1, \cdots, \nu_s \geq 1$  ($s \geq 1$); \ assume either \ $\zeta \geq1$ \ or \ $s \geq 2$, \ and \ set \ $\nu= \sum\limits_{j=1}^s \nu_j$ . Then

a)  \ \  $\Gamma_{(\zeta; \nu_1, \cdots, \nu_s) }$ is simply connected;

b)  \ \  $\pi_2 (\Gamma_{(\zeta; \nu_1, \cdots, \nu_s) })$ is isomorphic to $\mathbb{Z}^s$, if $\zeta \neq 2$, while $\pi_2 (\Gamma_{(2; \nu_1, \cdots, \nu_s) })$ is isomorphic to $\mathbb{Z}^{s+1}$;

c)  \ \  if $\nu_1 = \cdots = \nu_s =1$ and $\zeta=0, 1, 2$,\ \ then $\pi_i (\Gamma_{(\zeta; 1, \cdots, 1) })$ is isomorphic to \\ $\pi_i (S\mathcal{O}(\zeta+2s))$, for every $i \geq 3$.
\end{prop} 

\begin{proof}
The assumptions about $\zeta$ and $s$ imply that $\zeta+2\nu \geq3$; then $\pi_1 (S\mathcal{O}(\zeta+2\nu))$ is a cyclic group of order two.
Furthermore, since $\pi_2 (S\mathcal{O}(\zeta+2\nu))= \lbrace0\rbrace$, the final part of the exact homotopy sequence reduces to \ \ 
$0 \to \pi_2 (\Gamma_{(\zeta; \nu_1, \cdots, \nu_s) }) \overset{\delta} {\longrightarrow} \\
\pi_1 \big(S\mathcal{O}(\zeta)\oplus(\bigoplus\limits_{h=1}^s U(\nu_{_h}))\big) \overset{\psi}{\longrightarrow}\pi_1 (S\mathcal{O}(\zeta+2\nu)) \overset {\xi}{\longrightarrow}
\pi_1 (\Gamma_{(\zeta; \nu_1, \cdots, \nu_s) }) \to 0$. In this sequence the homomorphism $\psi$ is induced by the inclusion determined by the decomplexification mapping. 
Now we set $\phi_1= 0$,  $\phi_j= \sum\limits_{r=1}^{j-1} \nu_r$, for $j=2,\cdots, s$, and we define, for $j=1, \cdots, s$, the following loops:

$\alpha_j: t \mapsto I_{_{\zeta}} \oplus I_{_{\phi_j}}\oplus (e^{2 \pi t {\bf i}})\oplus I_{_{(\nu- \phi_j -1)}} \in S\mathcal{O}(\zeta)\oplus\big(\bigoplus\limits_{h=1}^s U(\nu_{_h})\big)$,

$\beta_j: t \mapsto I_{_{(2\phi_j+\zeta)}}\oplus \left( \begin{smallmatrix} 
\cos(2 \pi t) & \ -\sin(2 \pi t) \\ 
\sin(2 \pi t)&  \ \ \ \cos(2 \pi t)
\end{smallmatrix} \right)\oplus I_{_{2(\nu- \phi_j -1)}} \in S\mathcal{O}(\zeta+2\nu)$,
for every $t \in [0,1]$. Hence, denoted by $[\alpha_j]$ and $[\beta_j]$ the equivalence classes of the loops $\alpha_j$ and $\beta_j$ in $\pi_1\big(S\mathcal{O}(\zeta)\oplus(\bigoplus\limits_{h=1}^s U(\nu_{_h}))\big)$ and $\pi_1 (S\mathcal{O}(\zeta+2\nu))$, respectively, we have $\psi ([\alpha_j])= [\beta_j]$, for every $j=1, \cdots, s$. The mapping $\psi$ is surjective, since $[\beta_1]$ is the generator (of order two) of $\pi_1 (S\mathcal{O}(\zeta+2\nu))$; so, by the exactness of the previous sequence, $\pi_1 (\Gamma_{(\zeta; \nu_1, \cdots,\nu_s) })$ is the trivial group and (a) is proved.

Moreover, all the loops $\beta_j$ are homotopic to the loop $\beta_1$. Indeed, if $Q_j$ is a (special orthogonal) permutation matrix such that $Q_j \beta_{1}(t) Q_j ^T =  \beta_{j}(t)$ (for every $t \in [0,1]$) and $ \gamma: [0,1] \to S\mathcal{O}(\zeta+2\nu)$ is a continuous path joining $I_{_{(\zeta+2\nu)}}$ and $Q_j$, then the mapping $H$ defined by $H(t,s)= \gamma (s) \beta_{1} (t)  \gamma (s)^T$ (with $t,s \in [0,1]$), is a homotopy between the loops $\beta_{1}$ and $\beta_{j}$.Then $\psi ([\alpha_j])= [\beta_j]= [\beta_1]$, for $j=1, \cdots, s$. 

If $\zeta=0, 1$, the fundamental group $\pi_1 \big(S\mathcal{O}(\zeta)\oplus(\bigoplus\limits_{h=1}^s U(\nu_{_h}))\big)$ is a free abelian group of rank $s$ and its generators are the homotopy classes of the loops 
$\alpha_j$ for $j=1, \cdots, s$, so we have  $\psi(\sum\limits_{j=1}^s n_j [\alpha_j])= (\sum\limits_{j=1}^s n_j)[\beta_1]$, for every $n_1,\cdots,n_s \in \mathbb{Z}$. 

Furthermore $ \ker \psi = \lbrace \sum\limits_{j=1}^s n_j [\alpha_j]: \sum\limits_{j=1}^s n_j$ is even$\rbrace$ is a free abelian group, whose rank is less than or equal to $s=$ rank$\big(\pi_1 (S\mathcal{O}(\zeta)\oplus(\bigoplus\limits_{h=1}^s U(\nu_{_h})))\big)$. Since the elements $2[\alpha_1], \cdots, 2[\alpha_s] \in \ker \psi$ are linearly independent over $\mathbb{Z}$, it follows that rank$(\ker \psi) = s$ and then $\pi_2 (\Gamma_{(\zeta; \nu_1, \cdots,\nu_s) })\cong \ker \psi\cong \mathbb{Z}^s$.

If $\zeta \geq 2$, we denote now by 

$\omega: [0,1] \to S\mathcal{O}(\zeta)\oplus \big(\bigoplus\limits_{h=1}^s U(\nu_{_h})\big)$ and 
$\tilde{\omega}: [0,1] \to S\mathcal{O}(\zeta+2\nu)$ the loops defined, respectively, by 

$\omega (t)= \left( \begin{smallmatrix} 
\cos(2 \pi t) & \ -\sin(2 \pi t) \\ 
\sin(2 \pi t)&  \ \ \ \cos(2 \pi t)
\end{smallmatrix} \right) \oplus I_{_{(\zeta-2)}} \oplus I_{_\nu}$ and  $\tilde{\omega} (t)= \left( \begin{smallmatrix} 
\cos(2 \pi t) & \ -\sin(2 \pi t) \\ 
\sin(2 \pi t)&  \ \ \ \cos(2 \pi t)
\end{smallmatrix} \right) \oplus I_{_{(2\nu+\zeta -2)}}$, for every $t \in [0,1]$, so that $\psi([\omega])= [\tilde{\omega}]$. As before, it can be proved that 

$[\tilde{\omega}]= [\beta_1]$. 
Furthermore the elements $[\omega], [\alpha_1], \cdots, [\alpha_s]$ are independent generators of 
$\pi_1 \big(S\mathcal{O}(\zeta)\oplus (\bigoplus\limits_{h=1}^s U(\nu_{_h}))\big)$.

If $\zeta=2$, all these elements have infinite order and
$\pi_1 \big(S\mathcal{O}(\zeta)\oplus (\bigoplus\limits_{h=1}^s U(\nu_{_h}))\big)$ is a free abelian group of rank $s+1$. 

Then $\ker\psi=\lbrace n_0[\omega]+\sum\limits_{j=1}^s n_j [\alpha_j]: \sum\limits_{j=0}^s n_j $ is even$\rbrace$ is a free abelian group of rank $\leq s+1$; 
since $2[\omega], 2[\alpha_1], \cdots, 2[\alpha_s]$ are  $\mathbb{Z}$-linearly independent elements of $\ker \psi$, we conclude that rank$(\ker \psi) = s+1$ and then
$\pi_2 (\Gamma_{(2; \nu_1, \cdots, \nu_s) })\cong \ker \psi\cong \mathbb{Z}^{s+1}$. 

If $\zeta \geq 3$, we have $\pi_1 \big(S\mathcal{O}(\zeta)\oplus (\bigoplus\limits_{h=1}^s U(\nu_{_h}))\big)\cong \mathbb{Z}_2 \oplus \mathbb{Z}^s$ and hence 

rank \big($\pi_1 (S\mathcal{O}(\zeta)\oplus (\bigoplus\limits_{h=1}^s U(\nu_{_h})))$\big) $=s$. As before, $2[\alpha_1], \cdots, 2[\alpha_s]$ are $\mathbb{Z}$-linearly independent elements of $\ker \psi \subset\pi_1 \big(S\mathcal{O}(\zeta)\oplus (\bigoplus\limits_{h=1}^s U(\nu_{_h}))\big)$. Hence $s \leq $ rank($\ker \psi$) $\leq$ rank\big($\pi_1 (S\mathcal{O}(\zeta)\oplus (\bigoplus\limits_{h=1}^s U(\nu_{_h})))$\big) $=s$, so that rank($\ker \psi$) $=s$. Note that $\ker \psi$ is a torsion-free finitely generated abelian group. Indeed $[\omega]$ is the unique non-trivial torsion element of the group $\pi_1 \big(S\mathcal{O}(\zeta)\oplus (\bigoplus\limits_{h=1}^s U(\nu_{_h}))\big)$ and $[\omega] \notin \ker \psi$. Hence we obtain that $\ker \psi$ is a free abelian group of rank $s$, so that 

$\pi_2 (\Gamma_{(\zeta; \nu_1, \cdots, \nu_s) }) \cong \ker \psi \cong \mathbb{Z}^s$. Then (b) is completely proved.

Finally, if $\nu_j =1$, for every $j=1, \cdots, s$ and $\zeta=0, 1, 2$, since $\pi_r (U(1)) = \pi_r (S\mathcal{O}(1)) =\pi_r (S\mathcal{O}(2)) = \lbrace0\rbrace$ for every $r \geq 2$, we get the following exact sequence:

$0 \to \pi_i (S\mathcal{O}(\zeta+2\nu)) \overset {\xi}{\longrightarrow}\pi_i (\Gamma_{(\zeta; 1, \cdots, 1)}) \to 0$, for every $i \geq 3$, and so (c) holds.
\end{proof}

\begin{prop}\label{Theta-omotopia}
Let  $\nu_1, \cdots, \nu_s $ be integers such that $\nu_1, \cdots, \nu_s \geq 1$ \ ($s \geq 1$) and set
\ \ $\nu= \sum\limits_{j=1}^s \nu_j$. Then

a)  \ \  $\Theta_{(\nu_1, \cdots,\nu_s)}$ is simply connected;

b)  \ \  $\pi_2 (\Theta_{(\nu_1, \cdots, \nu_s)})$ is a free abelian group of rank $s-1$\ ;

c)  \ \  if $\nu_1 = \cdots = \nu_s =1$, then $\pi_i (\Theta_{(1, \cdots, 1) })$ is isomorphic to $\pi_i (U(s))$, for every $i \geq 3$. In particular, \ if $s \geq 2$ \ then
$\pi_3 (\Theta_{(1, \cdots, 1) })$ is isomorphic to $\mathbb{Z}$.
\end{prop}

\begin{proof}
Taking into account that $\pi_2(U(\nu))= \lbrace0\rbrace$ and arguing as in Proposition \ref{Gamma-omotopia}, we obtain (a); consequently we have the following short exact sequence:

$0 \to \pi_2 (\Theta_{(\nu_1, \cdots,\nu_s)}) \overset{\delta} {\longrightarrow}
\pi_1 (\bigoplus\limits_{j=1}^s U(\nu_j)) \overset{\psi}{\longrightarrow}\pi_1 (U(\nu)) \to 0$.

The group $\pi_2 (\Theta_{(\nu_1, \cdots,\nu_s) })$ is free abelian, since it is a subgroup of the free abelian group $\pi_1 (\bigoplus\limits_{j=1}^s U(\nu_j))\cong \mathbb{Z}^s$. Furthermore, the previous short exact sequence splits, because $\pi_1 (U(\nu))\cong \mathbb{Z}$ is a free abelian group. So, we can conclude that 

$\pi_2 (\Theta_{(\nu_1, \cdots,\nu_s)})\cong \mathbb{Z}^{s-1}$ and then (b) holds.

Since $\pi_r (U(1))=\lbrace0\rbrace$, for every $r \geq 2$, the exactness of the sequence 

$0 \to \pi_i (U(s)) \overset {\xi}{\longrightarrow}\pi_i (\Theta_{(1, \cdots, 1)}) \to 0$ implies that $\pi_i (\Theta_{(1, \cdots, 1)})\cong\pi_i (U(s))$, if $i \geq 3$. The last claim of (c) follows from the well-known fact that $\pi_3 (U(s))\cong \mathbb{Z}$, for every $s \geq 2$. 

\end{proof}

\begin{rem}\label{Pi2-con delta}

Using the Kronecker delta and taking into account Proposition \ref{omotopie-isomorfe}, it is possible to summarize Remark \ref{Quozienti} and Propositions, \ref{Bott}, \ref{Gamma-omotopia} (b), \ref{Theta-omotopia} (b), by saying that $\pi_2(\widehat{\Gamma}_{(\zeta; \nu_1, \cdots, \nu_s)})$ and $\pi_2(\Gamma_{(\zeta; \nu_1, \cdots, \nu_s)})$ are free abelian groups of rank 

$s-\delta_{(\zeta,0)}\delta_{(s,1)}\delta_{(\nu_{_1},1)} + \delta_{(\zeta,2)}(1-\delta_{(s,0)})$ \ (for $\zeta, s \geq 0$, $\zeta + s \geq 1$), while $\pi_2(\widehat{\Theta}_{(\nu_1, \cdots, \nu_s)})$ and $\pi_2(\Theta_{(\nu_1, \cdots, \nu_s)})$ are free abelian groups of rank $s-1$ , for every $s \geq 1$\ . Furthermore all these homogeneous spaces are simply connected.

\end{rem}
\section{Real logarithms of semi-simple matrices}\label{LOG}

Given a real square matrix $M$ of order $n$, we denote by $\mathcal{L}og(M)$ the set of all real logarithms of $M$, namely
$$\mathcal{L}og(M)=\lbrace Y \in M(n,\mathbb{R}) : exp(Y)=M \rbrace .$$
Let $M$ be semi-simple; in order for $\mathcal{L}og(M)$ not to be empty, by Remark \ref{esist-rad-quadr}, we must assume that the matrix $M$ is non-singular, with all (possible) negative eigenvalues having even multiplicity. 

Aim of this section is to study $\mathcal{L}og(M)$, for any $M$ satisfying these assumptions.

\begin{rem}\label{semisemplici}
Let $A \in M(n,\mathbb{R})$. Then $A$ is semi-simple if and only if $exp(A)$ is semi-simple.

\smallskip

One implication is trivial.
For the other, assume that $exp(A)$ is semi-simple. By the additive Jordan-Chevalley decomposition, there are a semi-simple matrix $S$ and a nilpotent matrix $N$ of index $k \geq 1$, such that $A= S+N$ with $SN=NS$ (see for instance \cite[\S\,4.2]{Humph1975} and also \cite{DoPe2017}). 
Since $N$ and $S$ commute, we have  $exp(A)=exp(S)exp(N)$ with $exp(S)$ semi-simple and $exp(N)$ unipotent, so, from the uniqueness of the multiplicative Jordan-Chevalley decomposition, we get $exp(A)=exp(S)$ and $exp(N)=I_n$. If $k \geq2$, from the latter equation we get $\sum\limits_{i=1}^{k-1} \dfrac{N^i}{i!}=0$ and this is impossible, since the degree of the minimal polynomial of $N$ is $k$; \ so, we necessarily have  $k=1$, $N=0$ and hence $A$ is semi-simple.
\end{rem}

\begin{remsdefs}\label{Jordan-form}

a) Let $M \in GL ^+(n,\mathbb{R})$ be a semi-simple matrix, whose every (possible) negative eigenvalue has even multiplicity, and denote its eigenvalues in the following way: 

\smallskip

- the distinct positive eigenvalues are: $\lambda_{_1} < \lambda_{_2} < \cdots < \lambda_{_p}$,\ with (positive) multiplicities \ $h_{_1}, h_{_2}, \cdots , h_{_p}$\ , respectively \ ($p \geq 0$);

\smallskip

- the distinct non-real eigenvalues are: $\rho_{_{(1,1)}}\exp(\pm {\bf i} \theta_{_1}), \cdots , \rho_{_{(1,a_{_1})}}\exp(\pm {\bf i} \theta_{_1})$, 

$\rho_{_{(2,1)}}\exp(\pm {\bf i} \theta_{_2}), \cdots , \rho_{_{(2,a_{_2})}}\exp(\pm {\bf i} \theta_{_2})$ \ up to \ 
$\rho_{_{(r,1)}}\exp(\pm {\bf i} \theta_{_r}), \cdots , \rho_{_{(r,a_{_r})}}\exp(\pm {\bf i} \theta_{_r})$, 

where $\rho_{_{(l,t)}} \exp(\pm {\bf i} \theta_{_l})$ have both (positive) multiplicity $m_{_{(l,t)}}$, for every $l,t$ , and  where $0 < \theta_{_1} < \theta_{_2} < \cdots < \theta_{_r} < \pi$ ,\ \ \  $a_{_l}\geq1$ ,\ \ \ $0 < \rho_{_{(l,1)}} < \rho_{_{(l,2)}} < \cdots  < \rho_{_{(l,a_{_l})}},$ for every $l = 1, \cdots , r$ \ \ ($r \ge 0$);

\smallskip

- the distinct negative eigenvalues are: $-w_{_1} > -w_{_2} > \cdots > -w_{_q}$,\  with (even positive) multiplicities \ $2 k_{_1} , 2 k_{_2}, \cdots , 2k_{_q}$\ ,\ respectively ($q \ge 0$)).

\smallskip

Note that\ \ \ \ \  $\sum\limits_{i=1}^p h_{_i} + 2 \sum\limits_{l=1}^r \ \sum\limits_{t=1}^{a_{_l}} m_{_{(l,t)}} + 2\sum\limits_{j=1}^q k_{_j} =n$.

We also denote by $2A=2\sum\limits_{l=1}^r a_{_l}$ \ the number of distinct non-real eigenvalues of $M$.
\smallskip

We have assumed that one or two of the indices $p, r, q$ can be zero. For example, the index $p$ vanishes when the matrix $M$ has no positive eigenvalues. In this case, the numbers $\lambda_{_i}, h_{_i}$ are not defined and it is understood that the term $\sum\limits_{i=1}^p h_{_i}$\ , or any other term of the same type, does not appear in the previous or similar equalities, in accordance with Notations \ref{notazioni} (b). Analogous remarks hold when $r$ or $q$ are zero. 
\smallskip

b) Let $M$ be as in (a) and denote by $Y \in M(n,\mathbb{R})$ a real logarithm of $M$. By Remark \ref{semisemplici}, $Y$ is semi-simple and its eigenvalues are (complex) logarithms of the eigenvalues of $M$. Hence there exist two finite sets, $ \lbrace \eta_{_{(i,x)}}, \tau_{_{(l,t,z)}}, \sigma_{_{(j,y)}} \rbrace \subset \mathbb{Z} $ and  $\lbrace u_{_{(i,x)}}, b_{_i}, \mu_{_{(l,t,z)}}, d_{_{(l,t)}}, v_{_{(j,y)}}, c_{_j} \rbrace \subset \mathbb{N}$, such that the distinct eigenvalues of $Y$ are precisely the following:

\smallskip

- $\ln(\lambda_{_i})\pm  2\pi\eta_{_{(i,x)}}{\bf i}$ \ , \ for  $x=0,1, \cdots ,b_{_i}$\ , \ where we can assume  

$0=\eta_{_{(i,0)}} < \eta_{_{(i,1)}} <\cdots<\eta_{_{(i,b_{_i})}},$ and where, if $b_{_i} \geq 1$ and $x=1, \cdots ,b_{_i}$, then the eigenvalues $\ln(\lambda_{_i})\pm  2\pi\eta_{_{(i,x)}}{\bf i}$  have both multiplicity $u_{_{(i,x)}} \ge 1$, \  while,\ for $x=0$,\ the multiplicity of $\ln(\lambda_{_i})$ \ is \ 
$g_{_i}:= h_{_i} - 2\sum\limits_{x=1}^{b_{_i}} u_{_{(i,x)}} \geq 0$,\ if $b_{_i} \geq 1$, \ and \ $g_{_i}:= h_{_i}$, if\ $b_{_i}=0$\ , for every\ $i=1, \cdots ,p$ ;

\smallskip

- $\ln(\rho_{_{(l,t)}}) \pm  (\theta_{_l}+ 2\pi \tau_{_{(l,t,z)}}){\bf i}$ ,\ \ both with multiplicity $\mu_{_{(l,t,z)}} \geq 1$,, \ for $z=1, \cdots, d_{_{(l,t)}}$, where $\tau_{_{(l,t,1)}} <\cdots< \tau_{_{(l,t,d_{_{(l,t)}})}}$ and $\sum\limits_{z=1}^{\ \ d_{_{(l,t)}}} \mu_{_{(l,t,z)}}= m_{_{(l,t)}}$\ , for every \ $t=1,\cdots,a_{_l}$ \ and  \ $l=1, \cdots, r$  
\  ;

\smallskip 

- $ln (w_{_j}) \pm(\pi + 2\pi \sigma_{_{(j,y)}}) {\bf i}$ ,\ \ both with multiplicity $v_{_{(j,y)}} \geq 1$,   , \ for $y = 1 , \cdots , c_{_j}$, \ where 
$v_{_{(j,1)}}< \cdots < v_{_{(j,c_{_j})}}$ \ and \ $\sum\limits_{y=1}^{c_{_j}} v_{_{(j,y)}}=k_{_j}$\ ,\ for every\ $j=1, \cdots ,q$\ .

\smallskip

c) Let $M,\ Y$ (and their eigenvalues) be as in (a) and (b), respectively. 

In order to simplify notations and statements, we define the following sets: 

$I:=\lbrace i:\ 1 \leq i \leq p\ , \ \ b_{_i}\geq1\rbrace , \ \ \ \ \ \ \ \ \  \widehat{I}:=\ \lbrace i: 1 \leq i \leq p \ , \ \ b_{_i}=0 \rbrace , \\ J := \lbrace i \in I:  \ \ g_{_i} = 0\rbrace = \lbrace i:\ 1 \leq i \leq p  \ ,\ \ g_{_i} = 0\rbrace , \ \ \ \ \ \ \ \widehat{J} := \lbrace i \in I: \ \ g_{_i} = 2\rbrace , \\ K= \lbrace i \in I:\ g_{_i}=0\ , \ \ b_{_i}=u_{_{(i,1)}}=1 \rbrace , \ \ L=\lbrace j:\ 1 \leq j \leq q\ , \ \ c_{_j}=v_{_{(j,1)}}=1\rbrace,$ 

and the following multi-indices:

$\eta:=(0\ ,\eta_{_{(1,1)}}, \cdots ,\eta_{_{(1,b_{_1})}}; \ \cdots \ ; 0\ , \eta_{_{(p,1)}}, \cdots, \eta_{_{(p,b_{_p})}});\\ 
u=:(g_{_1}, u_{_{(1,1)}}, \cdots , u_{_{(1,b_{_1})}}; \  \cdots \ ; g_{_p}, u_{_{(p,1)}}, \cdots , u_{_{(p,b_{_p})}});\\ $
$\tau:=(\tau_{_{(1,1,1)}},\cdots, \tau_{_{(1,1,d_{_{(1,1)}})}}; \ \cdots \ ; \tau_{_{(r,a_{_r},1)}}, \cdots , \tau_{_{(r,a_{_r},d_{_{(r,a_{_r})}})}});$

$\mu:=(\mu_{_{(1,1,1)}},\cdots, \mu_{_{(1,1,d_{_{(1,1)}})}}; \ \cdots \ ;\mu_{_{(r,a_{_r},1)}}, \cdots ,\mu_{_{(r,a_{_r},d_{_{(r,a_{_r})}})}});\\
\sigma:=(\sigma_{_{(1,1)}}, \cdots , \sigma_{_{(1,c_{_1})}}; \ \cdots \ ; \sigma_{_{(q,1)}}, \cdots \ ,\sigma_{_{(q,c_{_q})}}); \\ v:=(v_{_{(1,1)}}, \cdots , v_{_{(1,c_{_1})}}; \ \cdots \ ; v_{_{(q,1)}}, \cdots \ ,v_{_{(q,c_{_q})}})$.

If the set of multi-indices $(\eta, u, \tau, \mu, \sigma, v) $ satisfies the conditions stated in (b), we  say that it is \emph{admissible with respect to the matrix $M$} or simply \emph{M-admissible}.

Note that some multi-indices between $\eta, u, \tau, \mu, \sigma, v$ are necessarily empty when $p$, $r$ or $q$ vanish. For instance, if $p=0$, then $\eta= u = \emptyset ,$ and something like this when the integer r (or the integer q) is zero. 

Note also that there is always a countable infinity of $M$-admissible sets of multi-indices, unless the eigenvalues of $M$ are all real, positive and simple, in which case there exists a single $M$-admissible set of multi-indices corresponding to the values $\tau= \mu= \sigma= v = \emptyset, \ \ \eta=(0, 0, \cdots, 0) ,\ \ u=(h_{_1}, h_{_2}, \cdots, h_{_p})\ .$

We denote by 
$\mathcal{L}(M)_{(\eta, \tau, \sigma)}^{(u, \mu , v)}$ the subset of $\mathcal{L}og(M)$ of all real logarithms of $M$ whose eigenvalues agree with the eigenvalues of the matrix $Y$ (with the same multiplicities). We say that the eigenvalues of $Y$ (each with its own multiplicity) are the \emph{eigenvalues (with corresponding multiplicity) of $\mathcal{L}(M)_{(\eta, \tau, \sigma)}^{(u, \mu , v)}$}. 
Note also that, unless the eigenvalues of $M$ are all real, positive and simple, we have
$$\mathcal{L}og(M) = \bigsqcup \ \mathcal{L}(M)_{(u, \mu , v)}^{(\eta, \tau, \sigma)}\ \ ,$$ where the countable disjoint union is taken on all $M$-admissible sets of multi-indices $(\eta, u, \tau, \mu, \sigma, v)$, while, if the eigenvalues of $M$ are all real, positive and simple, the set $\mathcal{L}og(M)$ agrees with $\mathcal{L}(M)_{(u, \mu , v)}^{(\eta, \tau, \sigma)}$, where $\tau= \mu= \sigma= v = \emptyset, \ \ 
\eta=(0, 0, \cdots, 0) ,$ $\ \ u=(h_{_1}, h_{_2}, \cdots, h_{_p})$ \ .
It is not difficult to prove that each $\mathcal{L}(M)_{(\eta, \tau, \sigma)}^{(u, \mu , v)}$ is an open and closed topological subspace of $\mathcal{L}og(M)$ and, consequently, every connected component of $\mathcal{L}(M)_{(\eta, \tau, \sigma)}^{(u, \mu , v)}$ is also a connected component of $\mathcal{L}og(M)$.
\smallskip

d) If the semi-simple matrices $M$ and $Y$ (and their eigenvalues) are as in (a) and (b), the real Jordan forms, $\mathcal{J}_M$ of $M$ and $\widetilde{\mathcal{J}}$ of $J$, can be written, respectively, as follows:

$(\star) \ \ \ \ \ \mathcal{J}_M:= \bigg[ \bigoplus\limits_{i \in I} \lambda_{_i} I_{_{h_{_i}}}\bigg] \oplus \bigg[ \bigoplus\limits_{i \in \widehat{I}} \lambda_{_i} I_{_{h_{_i}}}\bigg] \oplus \bigg[\bigoplus\limits_{l=1}^r \ \bigoplus\limits_{t=1}^{a_{_l}} \rho_{_{(l,t)}} E_{\theta_{_l}}^{\oplus m_{_{(l,t)}}}\bigg] \oplus \bigg[ \bigoplus\limits_{j=1}^{q}(- w_{_j}) I_{_{2 k_{_j}}} \bigg]
$;

$(\star \star) \ \ \ \ \ \widetilde{\mathcal{J}}:= \bigg[\bigoplus\limits_{i\in I} \ \bigg((ln(\lambda_{_i})I_{_{g_{_i}}}) \oplus \big(\bigoplus\limits_{x=1}^{b_{_i}}(\ln(\lambda_{_i})I_{_{2u_{_{(i,x)}}}}+(2\pi\eta_{_{(i,x)}})E^{\oplus u_{_{(i,x)}}})\big)\bigg)\bigg]\oplus $
\smallskip

$\ \ \ \bigg[\bigoplus\limits_{i\in \widehat{I}} \ \ln(\lambda_{_i})I_{_{h_{_i}}}\bigg] \oplus \bigg[\bigoplus\limits_{l=1}^r \ \bigoplus\limits_{t=1}^{a_{_l}} \ \bigoplus\limits_{z=1}^{\ \ d_{_{(l,t)}}}\bigg(ln(\rho_{_{(l,t)}})I_{_{2\mu_{_{(l,t,z)}}}} + (\theta_{_l}+ 2\pi \tau_{_{(l,t,z)}})E^{\oplus \mu_{_{(l,t,z)}}} \bigg)\bigg] \oplus$
\smallskip 

$\ \ \ \bigg[\bigoplus\limits_{j=1}^q \ \bigoplus\limits_{y=1}^{c_{_j}} \bigg(ln(w_{_j})I_{_{2v_{_{(j,y)}}}} + (\pi + 2\pi \sigma_{_{(j,y)}})E^{\oplus v_{_{(j,y)}}}\bigg)\bigg]$.
\smallskip

By \cite[Cor.\,3.4.1.10, p.\,203]{HoJ2013}, we know that there exist two matrices 

$C, T \in GL(n,\mathbb{R})$ such that
$M= C \mathcal{J}_M C^{-1}$ and
$
Y=T \widetilde{\mathcal{J}} T^{-1}.
$

Since $\widetilde{\mathcal{J}}$ is a real Jordan form common to all matrices of $\mathcal{L}(M)_{(u, \mu , v)}^{(\eta, \tau, \sigma)}$, we say that $\widetilde{\mathcal{J}}$ is a \emph{real Jordan form of $\mathcal{L}(M)_{(u, \mu , v)}^{(\eta, \tau, \sigma)}$}.

Taking into account Notations \ref{notazioni} $(\bigtriangleup)$, it is easy to check that we have
 
$exp(\widetilde{\mathcal{J}}) = \mathcal{J}_M$ . 
Note also that this equality implies that $\mathcal{C}_{\widetilde{\mathcal{J}}} \subseteq \mathcal{C}_{{\mathcal{J}}_M}$.

\end{remsdefs}

\begin{prop}\label{descr-SR(M)-indici}
Let $M \in GL^+(n,\mathbb{R})$ be a semi-simple matrix, whose eigenvalues are as in Remarks-Definitions \ref{Jordan-form} (a) and let $C \in GL(n,\mathbb{R})$ be such that \\$M= C \mathcal{J}_M C^{-1}$, where $\mathcal{J}_M$ is the matrix defined by equation $(\star)$. 
Fix any set of \\$M$-admissible multi-indices $(\eta, u, \tau, \mu, \sigma, v)$ as in Remarks-Definitions \ref{Jordan-form} (b),(c) and denote by $\widetilde{\mathcal{J}}$ the real Jordan form of \ $\mathcal{L}(M)_{(u, \mu , v)}^{(\eta, \tau, \sigma)}$ defined by $(\star \star)$. 

Then we have \ \ $\mathcal{L}(M)_{(u, \mu , v)}^{(\eta, \tau, \sigma)} = \{C X \widetilde{\mathcal{J}} X^{-1} C^{-1} : X \in \mathcal{C}_{\mathcal{J}_M}\} = \Lambda_{_C}(\mathcal{L}(J_M)_{(u, \mu , v)}^{(\eta, \tau, \sigma)})$.

\smallskip

Moreover, $\mathcal{L}(M)_{(u, \mu , v)}^{(\eta, \tau, \sigma)}$ is a closed embedded submanifold of $GL(n,\mathbb{R})$, diffeomorphic to the  homogeneous space $\dfrac{\mathcal{C}_{\mathcal{J}_M}}{\mathcal{C}_{\widetilde{{\mathcal{J}}}}}$.
\end{prop}
\begin{proof}
If $Y \in \mathcal{L}(M)_{(u, \mu , v)}^{(\eta, \tau, \sigma)}$, then $Y =T \widetilde{\mathcal{J}} T^{-1}$ for some $T \in GL(n,\mathbb{R})$ and $exp(Y)=T exp(\widetilde{\mathcal{J}}) T^{-1} = T \mathcal{J}_M T^{-1} = M = C \mathcal{J}_M C^{-1} $. Hence $C^{-1}T \in \mathcal{C}_{{\mathcal{J}}_M}$ and, so, $T=CX$ for some $X \in \mathcal{C}_{{\mathcal{J}}_M}$. This gives one inclusion in the first equality of the statement. The reverse inclusion is a simple computation. The second equality of the statement follows directly from the definition of the mapping $\Lambda_{_C}$.
Since $\Lambda_{_C}$ is a diffeomorphism of $GL(n,\mathbb{R})$,  it suffices to prove that the set 
$\{ X \widetilde{\mathcal{J}} X^{-1}  : X \in \mathcal{C}_{\mathcal{J}_M}\} = \mathcal{L}(\mathcal{J}_M)_{(u, \mu , v)}^{(\eta, \tau, \sigma)}$ has the properties required for set $\mathcal{L}(M)_{(u, \mu , v)}^{(\eta, \tau, \sigma)}$.

Let us consider the left action by conjugation of $\mathcal{C}_{\mathcal{J}_M}$ on $GL(n,\mathbb{R})$. 
The set $\{ X \widetilde{\mathcal{J}} X^{-1}  : X \in \mathcal{C}_{\mathcal{J}_M}\}$ is the orbit of $\widetilde{\mathcal{J}}$. By Remark \ref{Popov}, this set is an immersed submanifold of $GL(n,\mathbb{R})$, diffeomorphic to the  homogeneous space $\dfrac{\mathcal{C}_{\mathcal{J}_M}}{\mathcal{C}_{\widetilde{{\mathcal{J}}}}}$, since $\mathcal{C}_{\widetilde{{\mathcal{J}}}}$ is the isotropy subgroup of the action. 

Finally $\mathcal{L}(\mathcal{J}_M)_{(u, \mu , v)}^{(\eta, \tau, \sigma)}$ is  closed in $GL(n,\mathbb{R})$. Indeed, if $\{Y_i\}_{_{i\in \mathbb{N}}}$ is a sequence in 

$\mathcal{L}(\mathcal{J}_M)_{(u, \mu , v)}^{(\eta, \tau, \sigma)}$, converging to $Y \in GL(n,\mathbb{R})$, then 
$exp(Y) =\mathcal{J}_M$ too, and the characteristic polynomial of $Y$ is the same characteristic polynomial of all $Y_i$'s (constant with respect to $i \in \mathbb{N}$). 
Hence 
$Y \in \mathcal{L}(\mathcal{J}_M)_{(u, \mu , v)}^{(\eta, \tau, \sigma)}$ and this last set is closed and, therefore, it is an embedded submanifold of $GL(n,\mathbb{R})$ (see for instance \cite[\S\,2.13  Theorem, p.\,65]{MonZip1955}).
\end{proof}

\begin{lemma}\label{lemma-su-J}

Let $\mathcal{J}_M$ and $\widetilde{\mathcal{J}}$ the matrices of Remarks-Definitions \ref{Jordan-form} (d) defined by $(\star)$ and $(\star \star)$, respectively. Then the Lie groups of non-singular matrices commuting with $\mathcal{J}_M$ and $\widetilde{\mathcal{J}}$ are the following:

$\mathcal{C}_{\mathcal{J}_M} = \bigg[\bigoplus\limits_{i \in I} GL(h_{_i},\mathbb{R})\bigg] \oplus \bigg[\bigoplus\limits_{i \in \widehat{I}} GL(h_{_i},\mathbb{R})\bigg] \oplus \bigg[\bigoplus\limits_{l=1}^r \ \bigoplus\limits_{t=1}^{a_{_l}} GL(m_{_{(l,t)}},\mathbb{C})\bigg] \oplus $
\smallskip

$ \ \ \ \ \ \bigg[\bigoplus\limits_{j=1}^q GL(2k_{_j},\mathbb{R})\bigg]$.
\smallskip

$\mathcal{C}_{\widetilde{\mathcal{J}}} = \bigg[\bigoplus\limits_{i\in I} \bigg(GL({g_{_i}},\mathbb{R})\oplus \big(\bigoplus\limits_{x=1}^{b_{_i}}GL({u_{_{(i,x)}}},\mathbb{C})\big)\bigg)\bigg]\oplus\bigg[\bigoplus\limits_{i\in \widehat{I}} GL({h_{_i}},\mathbb{R})\bigg] \oplus $
\smallskip

$\ \ \ \ \ \bigg[\bigoplus\limits_{l=1}^r \ \bigoplus\limits_{t=1}^{a_{_l}} \ \bigoplus\limits_{z=1}^{\ \ d_{_{(l,t)}}} GL({\mu_{_{(l,t,z)}}},\mathbb{C})\bigg] \oplus \bigg[\bigoplus\limits_{j=1}^q \ \bigoplus\limits_{y=1}^{c_{_j}} GL({v_{_{(j,y)}}},\mathbb{C})\bigg]$.
\end{lemma}

\begin{proof}
The statement follows directly from Lemmas \ref{commutazione} and \ref{Ethetacommutaz}.
\end{proof}

\begin{thm}\label{logaritmi} Let $M \in GL^+(n,\mathbb{R})$ be a semi-simple matrix, whose eigenvalues are as in Remarks-Definitions \ref{Jordan-form} (a). Fix any set $(\eta, u, \tau, \mu, \sigma, v)$  of $M$-admissible multi-indices, denote by $2A$ the number of distinct non-real eigenvalues of $M$ and define the sets $I, J, \widehat{J}, K,  L$ as in Remarks-Definitions \ref{Jordan-form} (c). Then 

a)\ \ $\mathcal{L}(M)_{(u, \mu , v)}^{(\eta, \tau, \sigma)}$ is a manifold with $2^{(|J|+q)}$ connected components, each of which is diffeomorphic to\ \  $ \bigg[ \prod\limits_{i \in I} \widehat{\Gamma}_{(g_{_i}; u_{_{(i,1)}}, \cdots, \ u_{_{(i,b_{_i})}})} \bigg] \times \bigg[\prod\limits_{l=1}^r \ \prod\limits_{t=1}^{a_{_l}} \widehat{\Theta}_{(\mu_{_{(l,t,1)}}, \cdots, \ \mu_{_{(l,t,d_{_{(l,t)}})}})} \bigg] \times$

\ \ \ $\bigg[ \prod\limits_{j=1}^q \widehat{\Gamma}_{(0; v_{_{(j,1)}}, \cdots, \ v_{_{(j,c_{_j})}})}  \bigg]$.
\smallskip

Denote by \ $\mathcal{C}$ an arbitrary connected component of $\mathcal{L}(M)_{(u, \mu , v)}^{(\eta, \tau, \sigma)}$. Then

b)\ \ $\mathcal{C}$ is simply connected and \ $\pi_2(\mathcal{C})$ \ is a free abelian group whose rank is

\ \ $\sum\limits_{i=1}^p b_{_i} \ - \ |K| \ + \ |\widehat{J}| \ + \ \sum\limits_{l=1}^r \ \sum\limits_{t=1}^{a_{_l}} d_{_{(l,t)}} \ - \ A \ + \ \sum\limits_{j=1}^q c_{_j} \ - \ |L|\ ;$ 
\smallskip

c)\ \ assume that all non-real eigenvalues of $\mathcal{L}(M)_{(u, \mu , v)}^{(\eta, \tau, \sigma)}$ are simple; hence the rank of \ $\pi_2(\mathcal{C})$ \ is \ \ $\dfrac{1}{2}\big(n-\sum\limits_{i=1}^p g_{_i}\big)\ - \ |K| \ + \ |\widehat{J}| \ - \ A - \ |L|\ .$ 

If, in addition, the multiplicity of all real eigenvalues of $\mathcal{L}(M)_{(u, \mu , v)}^{(\eta, \tau, \sigma)}$ is less than or equal to $2$, \ then, \ for every $\alpha \geq 3$,\ \ \ $\pi_{\alpha}(\mathcal{C})$ is isomorphic to the direct sum 
\smallskip

\ \ $\bigg[\bigoplus\limits_{i=1}^p \pi_{\alpha}\big(S\mathcal{O}(h_{_i})\big)\bigg]\oplus \bigg[ \bigoplus\limits_{l=1}^r \ \bigoplus\limits_{t=1}^{a_{_l}} \pi_{\alpha}\big(U(m_{_{(l,t)}})\big) \bigg] \oplus \bigg[ \bigoplus\limits_{j=1}^q \pi_{\alpha}\big(S\mathcal{O}(2k_{_j})\big) \bigg]$.
\end{thm} 
\begin{proof}
From Proposition \ref{descr-SR(M)-indici} and from Lemma \ref{lemma-su-J}, it follows that $\mathcal{L}(M)_{(u, \mu , v)}^{(\eta, \tau, \sigma)}$ is a manifold diffeomorphic to the following product of homogeneous spaces:

\smallskip
$\bigg[ \prod\limits_{i \in I}\dfrac{GL(h_{_i},\mathbb{R})}{GL(g_{_i},\mathbb{R})\oplus \big(\bigoplus\limits_{x=1}^{b_{_i}}GL(u_{_{(i,x)}},\mathbb{C})\big)} \bigg] \times \bigg[\prod\limits_{l=1}^r \ \prod\limits_{t=1}^{a_{_l}} \dfrac{GL(m_{_{(l,t)}},\mathbb{C})}{\big( \bigoplus\limits_{z=1}^{d_{_{(l,t)}}} GL(\mu_{_{(l,t,z)}},\mathbb{C})\big)} \bigg] \times$
\smallskip

\ \ \ $ \bigg[ \prod\limits_{j=1}^q \dfrac{GL(2 k_{_j},\mathbb{R})}{\big(\bigoplus\limits_{y=1}^{c_{_j}} GL(v_{_{(j,y)}},\mathbb{C})\big)}  \bigg];$
\smallskip

hence, by Remarks-Definitions \ref{def-spaz-omog} (d), we can easily get the statement (a).
\smallskip

By (a) and Remark \ref{Pi2-con delta}, we get that the component $\mathcal{C}$ is simply connected and the rank of the free abelian group $\pi_2(\mathcal{C})$ is 

$\sum\limits_{i \in I} \big( b_{_i}-\delta_{(g_{_i},0)}\delta_{(b_{_i},1)}\delta_{(u_{_{(i,1)}},1)}+\delta_{(g_{_i},2)} (1-\delta_{(b_{_i},0)})\big) \ + \ \sum\limits_{l=1}^r \ \sum\limits_{t=1}^{a_{_l}} d_{_{(l,t)}} - \ \sum\limits_{l=1}^r a_{_l} \ + \\ \sum\limits_{j=1}^q \big(c_{_j} - \delta_{(c_{_j},1)}\delta_{(v_{_{(j,1)}},1)}\big) = \sum\limits_{i=1}^p b_{_i} \ - \ |K| \ + \ |\widehat{J}|\ + \ \sum\limits_{l=1}^r \ \sum\limits_{t=1}^{a_{_l}} d_{_{(l,t)}} - \ A \ + \sum\limits_{j=1}^q c_{_j} \ -\ |L|\ ;$ 

so (b) is proved.

For part (c), we note that the condition on the non-real eigenvalues of $\mathcal{L}(M)_{(u, \mu , v)}^{(\eta, \tau, \sigma)}$ is equivalent to $u_{_{(i,x)}}= \mu_{_{(l,t,z)}}= v_{_{(j,y)}}=1$,\ \  for every possible choice of the indices $i,x,l,t,z,j,y$, \ so that, under this condition, we have

$(\bullet)\ \ \ \ h_{_i}=g_{_i}+2b_{_i}\ , \ \ m_{_{(l,t)}}= d_{_{(l,t)}}\ , \ \ k_{_j}=c_{_j}\ ,$ \ for all possible indices $i,l,t,j$.

Hence $\sum\limits_{i=1}^p b_{_i} +\sum\limits_{l=1}^r \ \sum\limits_{t=1}^{a_{_l}} d_{_{(l,t)}} +\sum\limits_{j=1}^q c_{_j}= \sum\limits_{i=1}^p \dfrac{(h_{_i}-g_{_i})}{2} +\sum\limits_{l=1}^r \ \sum\limits_{t=1}^{a_{_l}} m_{_{(l,t)}} +\sum\limits_{j=1}^q k_{_j}= \\ \dfrac{1}{2}\big(n- \sum\limits_{i=1}^p \ g_{_i}\big) .$ \ \ From this and from (b), we get the requested formula for the rank of $\pi_2(\mathcal{C})$ .
If, in addition, the condition on the real eigenvalues holds, then we have $g_{_i} \leq 2$, \ for $i=1, \cdots, p$, \ so the formula for $\pi_{\alpha}(\mathcal{C})$ \ $(\alpha \geq 3)$ follows from
Propositions \ref{omotopie-isomorfe}, \ \ref{Gamma-omotopia} (c) and \ \ref{Theta-omotopia} (c), taking into account the equalities $(\bullet)$ and that, if $i \notin I$ then $\pi_{\alpha}(S\mathcal{O}(h_{_i}))=\pi_{\alpha}(S\mathcal{O}(g_{_i}))=\lbrace0\rbrace$,\ for every $\alpha \geq 3$.
\end{proof}

\begin{rem}\label{numerab-log}
Let $M \in GL^+(n,\mathbb{R})$ be a semi-simple matrix whose negative eigenvalues have even multiplicity. By Remarks-Definitions \ref{Jordan-form} (c), Theorem \ref{logaritmi} (a) and Remarks-Definitions \ref{def-spaz-omog} (b), $\mathcal{L}og(M)$ is a finite set if and only if the eigenvalues of $M$ are all real, positive and simple, and in this case, it consists of a single point (see \cite[Thm.\,2]{Culv1966}). In all other cases, we have that the set $\mathcal{L}og(M)$ is countably infinite if and only if every manifold $\mathcal{L}(M)_{(h, m , k)}^{(\eta, \tau, \sigma)}$ has zero dimension, and so, taking into account Theorem \ref{logaritmi} (a) and Remarks-Definitions \ref{def-spaz-omog} (b),(c),
we get that the set $\mathcal{L}og(M)$ is countably infinite if and only if all eigenvalues of $M$ are simple and no eigenvalue of $M$ is negative, in accordance with \cite[Cor.\, p.\,1151]{Culv1966}.
\end{rem}

\begin{defi}\label{prin-gen}
Let $M \in M(n,\mathbb{R})$. We say that a matrix $X \in M(n,\mathbb{R})$ is a \emph{generalized principal real logarithm} of $M$, if $exp(X)=M$ and every eigenvalue of $X$ has imaginary part in $[-\pi, \pi]$.
This definition is more general than the usual one of \emph{principal real logarithm}  (see for instance \cite[Thm.\,1.31 p.\,20]{Hi2008}).

We denote by $\mathcal{PL}og(M)$ the set of all generalized principal real logarithms of $M$. Of course this set can be empty, but this does not happen if the matrix M is non-singular, semi-simple and all its negative eigenvalues have even multiplicity.
\end{defi}

\begin{thm}\label{logar-princ} Let $M \in GL^+(n,\mathbb{R})$ be a semi-simple matrix, whose distinct negative eigenvalues are exactly $q$ (where $q \geq0$) and have multiplicity $2k_{_1}, \cdots , 2k_{_q}$, respectively. If \ $q \geq 1$, \ the set \ $\mathcal{PL}og(M)$ is a manifold with $2^q$  connected components, each of which is diffeomorphic to the symmetric space $\prod\limits_{j=1}^q\widehat{\Gamma}_{(0 ; k_{_j})}$,  while, if $M$ has no negative eigenvalues, the set $\mathcal{PL}og(M)$ is a single point . Moreover, if $\mathcal{C}$ is any connected component of \ $\mathcal{PL}og(M)$, then $\mathcal{C}$ is simply connected and $\pi_2(\mathcal{C})$ is a free abelian group of rank \ $B$, where $B$ is the number of distinct negative eigenvalues of $M$, whose multiplicity is is greater than or equal to 4.
\end{thm}

\begin{proof}
Using the same notations as in Remarks-Definition \ref{Jordan-form}, we have \\$\mathcal{PL}og(M)=\mathcal{L}(M)_{(h, m , k)}^{(\eta, \tau, \sigma)}$ , where $\eta = \mathsf{O}$, $\tau = \mathsf{O}$, $\sigma= \mathsf{O}$, 
\ $\mu=(m_{_{(1,1)}};\ \cdots\ ; m_{_{(r,a_{_r})}})$, \ $u=(h_{_1}; \ \cdots ;h_{_p})$ \ $v=(k_{_1};\ \cdots\ ;k_{_q})$; here $\mathsf{O}$ indicates any multi-index whose  entries are all zero.
So, by Theorem \ref{logaritmi} (a), the manifold $\mathcal{PL}og(M)$ has $2^q$ connected components, which are, if $q \geq 1$, all diffeomorphic to $\prod\limits_{j=1}^q\widehat{\Gamma}_{(0 ; k_{_j})}$ , while $\mathcal{PL}og(M)$ consists of a single point, when $q=0$; indeed $\widehat{\Gamma}_{(h_{_i})}$ and $\widehat{\Theta}_{(m_{_{(l,t)}})}$ reduce to a single point, for all possible indices $i,l,t$. 

The final part of the statement follows from Theorem \ref{logaritmi} (b), taking into account that, in this case, \ the set $I$ is empty, \ \ \ $ d_{_{(l,t)}}=c_{_j}=1$ \ for all possible indices $l,t,j$, \\
$L=\lbrace j:\ 1 \leq j \leq q \ , \ \ k_{_j}=1 \rbrace$ \ and \ $B= q-|L|$ . 
\end{proof}

\section{Skew-symmetric real logarithms of special orthogonal matrices} \label{Skew-log}

In this Section we assume $n \geq 2$.
\begin{notas}\label{notaz_ortog} 
Let $M \in S\mathcal{O}(n)$. Since the eigenvalues of $M$ have unitary modulus, the real Jordan form of $M$ can be written as follows:
$$(*) \ \ \ \ \ \mathcal{\widehat{J}}_M= I_{_{h}} \oplus E_{\theta_{_1}}^{\oplus m_{_1}} \oplus \cdots \oplus E_{\theta_{_r}}^{\oplus  m_{_r}} \oplus(-I_{_{2k}})\ \ \ \ ,$$ 
where $h, r, k \ge 0$, \ \ \ $h+ 2 m_1 + \cdots + 2 m_r + 2k =n$, \  \ \ $0 < \theta_{_1} < \theta_{_2} < \cdots < \theta_{_r} < \pi$;\ \ so the eigenvalues of $M$ are: $1$ with multiplicity $h \ge 0$, \ $\exp(\pm {\bf i} \theta_{_1})$ both with multiplicity $m_{_1}$, $\cdots$ , up to $\exp(\pm {\bf i} \theta_{_r})$ both with multiplicity $m_{_r}$ ($m_{_j} \geq1\ ,\ $for every $j$, if $r > 0$) and $-1$ with multiplicity $2 k \ge 0$.

Note that, also in this case, the integers $h, r, k$ can vanish; so we assume, in this Section, the same agreements stated in Notations \ref{notazioni} (b).

Note also that, if $n$ is odd, then $h$ is also odd and, in particular, \ $1$ is necessarily an eigenvalue of $M$.

It is well known that there exists $Q \in \mathcal{O}(n)$ such that $M=Q\mathcal{\widehat{J}}_M Q^T$ (see for instance \cite[Cor.\,2.5.11 p.\,137]{HoJ2013}).
\end{notas}
\begin{defis}\label{skew_symmetric_logarithms}
a) 
\ If $V\in \mathfrak{so}(n)$ and $exp{(V)}=M$ , we say that $V$ is a \emph{ skew-symmetric real logarithm of $M$}.  We denote by 
$ \mathcal{L}og_{{\mathfrak{so}(n)}}(M) = \mathcal{L}og(M)\cap \mathfrak{so}(n)$ the set of skew-symmetric real logarithms of $M$.
Now fix $W\in \mathcal{L}og_{{\mathfrak{so}(n)}}(M)$. Since the eigenvalues of $W$ are complex logarithms of those of $M$, as in Remarks-Definitions \ref{Jordan-form}, there exist two finite sets, 

$\lbrace \eta_{_i}, \ \tau_{_{(l,t)}}, \ \sigma_{_j} \rbrace \subset \mathbb{Z}$ and $\lbrace u_{_i}, \ b, \ \mu_{_{(l,t)}}, \ d_{_l}, \ v_{_j}, \ c \rbrace \subset \mathbb{N}$, such that the eigenvalues of $W$ can be written in the following way:

- \ $\pm 2\pi \eta_{_i} \bf i$ , \ for $ i =0,1, \cdots, b$\ , \ with \ $0= \eta_{_0}<  \eta_{_1}< \cdots <\eta_{_b} $, \ so that, when $b \geq 1$ and $\ i = 1,\cdots , b \ ,$ \ the eigenvalues \ $\pm 2\pi \eta_{_i} \bf i$ \ have both multiplicity $u_{_i} \geq 1,$ \ while, for $i=0$, \ the eigenvalue \ $0$ \ has multiplicity $g:= h-2 \sum\limits_{i=1}^b u_{_i} \geq 0 $ \ if $b \geq 1$ \ and \ $g:= h$ \ if \ $b=0$;
\smallskip

- \ $\pm(\theta_{_l} + 2\pi \tau_{_{(l,t)}})\bf i$ \ both with multiplicity $\mu_{_{(l,t)}}\geq 1$ ,\ for every $ t=1 \cdots d_{_l} $,\ where $\tau_{_{(l,1)}} < \tau_{_{(l,2)}} < \cdots  < \tau_{_{(l,d_{_l})}}$ and $\sum\limits_{t=1}^{d_{_l}} \mu_{_{(l,t)}} = m_{_l}$\ ,\ for every $l=1, \cdots , r $ ;
\smallskip

- \ $\pm(\pi + 2\pi \sigma_{_j})\bf i$ \ both with multiplicity $v_{_j}\geq 1$, for every $j=1, \cdots ,c$\ , where $\sigma_{_1} <\cdots <\sigma_{_c} $ and $\sum\limits_{j=1}^c v_{_j}=k ,$  . 
\smallskip 

As in Remarks-Definitions \ref{Jordan-form} (b), in order to simplify the notations, we define the following multi-indices:

$\eta:=(0, \eta_{_1}, \cdots ,\eta_{_b}); \ \ \ \ \ \ \ u:=(g, u_{_1}, \cdots  ,u_{_b}); \\ \tau:=(\tau_{_{(1,1)}}, \cdots , \tau_{_{(1,d_{_1})}} ; \cdots ; \tau_{_{(r,1)}}, \cdots , \tau_{_{(r,d_{_r})}});\\
\mu:=(\mu_{_{(1,1)}},\cdots, \mu_{_{(1,d_{_1})}} ; \cdots ; \mu_{_{(r,1)}} , \cdots ,\mu_{_{(r,d_{_r})}});
\\ \sigma:=(\sigma_{_1}, \cdots , \sigma_{_c}); \ \ \ \ \ \ \ v:=(v_{_1}, \cdots , v_{_c});$

with the notations of Section \ref{LOG}, \ the set of multi-indices \ $(\eta, u, \tau, \mu, \sigma, v) $ is \\$M$-admissible. Note that, since $n \geq 2$, there is a countable infinity of $M$-admissible sets of multi-indices, for every $M \in S\mathcal{O}(n)$. Note also that $-1$ is an eigenvalue of $M$ of multiplicity $2$ if and only if \ $c=v_{_1}=1$, for every $M$-admissible set of multi-indices.

b) Let $W$ as in (a). The previous hypotheses on the eigenvalues of $W$ are equivalent to say that there exists a real Jordan form $\widehat{\mathcal{J}}$ of $W$ of the following type:

$(**)\ \ \widehat{\mathcal{J}}=$ 
$\bigg[\mathrm{O}_{_{g}}\oplus \big( \bigoplus\limits_{i=1}^b (2\pi\eta_{_i} E^{\oplus u_{_i}})\big)\bigg]\oplus$

\ \ \ \ \ \ \ \ \ \ \ \ \ \ \ \ \ \ \ \ \ \ \ \ \ \ \ \ \ \ $\oplus \bigg[ \bigoplus\limits_{l=1}^r \ \bigoplus\limits_{t=1}^{a_{_l}} (\theta_{_l}+2\pi \tau_{_{(l,t)}}) E^{\oplus \mu_{_{(l,t)}}}\bigg]\oplus \bigg[\bigoplus\limits_{j=1}^c (\pi+2\pi \sigma_{_j})E^{\oplus v_{_j}}\bigg].$ 

Since the matrix $W$ is skew-symmetric, there exists a matrix $C \in \mathcal{O}(n)$ such that
$W= C \widehat {\mathcal{J}} C^T$
(see again \cite[Cor.\,2.5.11 p.\,136]{HoJ2013}).
Note that $\widehat {\mathcal{J}}$ is skew-symmetric too.

c) If the set of multi-indices $(\eta, u, \tau, \mu, \sigma, v) $ is $M$-admissible , we denote by 

$\mathcal{L}_{\mathfrak{so}(n)}(M)_{(u, \mu , v)}^{(\eta, \tau, \sigma)}= \mathfrak{so}(n)\cap \mathcal{L}(M)_{(\eta, \tau, \sigma)}^{(u, \mu , v)}$ the set of skew-symmetric real logarithms of $M$ whose real Jordan form is the matrix $\widehat{\mathcal{J}}$ as in $(**)$. As in Remarks-Definitions \ref{Jordan-form} (c),(d), we say that $\widehat{\mathcal{J}}$ is a \emph{real Jordan form of $\mathcal{L}_{\mathfrak{so}(n)}(M)_{(u, \mu , v)}^{(\eta, \tau, \sigma)}$} and that its eigenvalues (with related multiplicities) are the \emph{eigenvalues (with corresponding multiplicities) of $\mathcal{L}_{\mathfrak{so}(n)}(M)_{(u, \mu , v)}^{(\eta, \tau, \sigma)}$}.
It is clear that $$\mathcal{L}og_{{\mathfrak{so}(n)}}(M)= \bigsqcup \mathcal{L}_{\mathfrak{so}(n)}(M)_{(u, \mu , v)}^{(\eta, \tau, \sigma)} \ \ ,$$
where the countable disjoint union is taken on all $M$-admissible sets of multi-indices $(\eta, u, \tau, \mu, \sigma, v) $.
As in Remarks-Definitions \ref{Jordan-form} (c), \ $\mathcal{L}_{\mathfrak{so}(n)}(M)_{(u, \mu , v)}^{(\eta, \tau, \sigma)} $ is an open and closed topological subspace of $\mathcal{L}og_{{\mathfrak{so}(n)}}(M)$, \ and so we get that each connected component of $\mathcal{L}_{\mathfrak{so}(n)}(M)_{(u, \mu , v)}^{(\eta, \tau, \sigma)} $ is a connected component of $\mathcal{L}og_{{\mathfrak{so}(n)}}(M)$ too, for every $M$-admissible set of multi-indices $(\eta, u, \tau, \mu, \sigma, v)$.
\end{defis}
\begin{prop}\label{log-ortog}
Let $M \in S \mathcal{O}(n)$ whose real Jordan form, $\mathcal{\widehat{J}}_M$, is as in Notations   \ref{notaz_ortog}  $(*)$ and let $Q \in \mathcal{O}(n)$ such that $M= Q \mathcal{\widehat{J}}_M Q^T$. 
Fix any $M$-admissible set of multi-indices $(\eta, u, \tau , \mu, \sigma, v)$ and let 
$\widehat{\mathcal{J}}$ be the real Jordan form of $\mathcal{L}_{\mathfrak{so}(n)}(M)_{(u, \mu , v)}^{(\eta, \tau, \sigma)}$ as in Definitions \ref{skew_symmetric_logarithms} $(**)$. Then
we have

$
\mathcal{L}_{\mathfrak{so}(n)}(M)_{(u, \mu , v)}^{(\eta, \tau, \sigma)} = \{ Q X \widehat{\mathcal{J}}   X^T Q^T : X \in  \mathcal{C}_{\mathcal{\widehat{J}}_M} \cap \mathcal{O}(n)\} =  \Lambda_{_Q}(\mathcal{L}_{\mathfrak{so}(n)}(\mathcal{\widehat{J}}_M)_{(u, \mu , v)}^{(\eta, \tau, \sigma)}) .
$

\smallskip

Moreover $\mathcal{L}_{\mathfrak{so}(n)}(M)_{(u, \mu , v)}^{(\eta, \tau, \sigma)}$ is a compact submanifold of $\mathfrak{so}(n)$, diffeomorphic to the homogeneous space $\dfrac{\mathcal{C}_{\mathcal{\widehat{J}}_M} \cap \mathcal{O}(n)}{\mathcal{C}_{\widehat{{\mathcal{J}}}} \cap \mathcal{O}(n)}$.
\end{prop}

\begin{proof}
Let $W \in \mathcal{L}_{\mathfrak{so}(n)}(M)_{(u, \mu , v)}^{(\eta, \tau, \sigma)}$. We know that there exists $C \in \mathcal{O}(n)$ such that
$W= C \widehat {\mathcal{J}} C^T.$
By Notations \ref{notazioni} $(\bigtriangleup)$, we get $exp(\widehat{\mathcal{J} })=\mathcal{\widehat{J}}_M .$ Hence, since $exp{(W)}=M$, we obtain $C\mathcal{\widehat{J}}_MC^T=Q\mathcal{\widehat{J}}_MQ^T$ and so $Q^TC=X$ commutes with $\mathcal{\widehat{J}}_M.$ \ So we have $C=QX$, with $X \in \mathcal{C}_{\mathcal{\widehat{J}}_M}\cap \mathcal{O}(n).$

Conversely, if $W= QX\widehat{\mathcal{J} }X^TQ^T$, with $X\in \mathcal{C}_{\mathcal{\widehat{J}}_M}\cap \mathcal{O}(n)$,  then $W$ is skew-symmetric and $exp(W)=QX\mathcal{\widehat{J}}_MX^TQ^T=Q\mathcal{\widehat{J}}_MQ^T=M.$ Hence we have 
$\mathcal{L}_{\mathfrak{so}(n)}(M)_{(u, \mu , v)}^{(\eta, \tau, \sigma)} =\\
 \{ Q X \widehat{\mathcal{J}}   X^T Q^T : X \in  \mathcal{C}_{\mathcal{\widehat{J}}_M} \cap \mathcal{O}(n)\} = \Lambda_{_Q}(\{ X \widehat{\mathcal{J}}   X^T  : X \in  \mathcal{C}_{\mathcal{\widehat{J}}_M} \cap \mathcal{O}(n)\})$. This concludes the first part, since we have $\{ X \widehat{\mathcal{J}}   X^T  : X \in  \mathcal{C}_{\mathcal{\widehat{J}}_M} \cap \mathcal{O}(n)\} = \mathcal{L}_{\mathfrak{so}(n)}(\mathcal{\widehat{J}}_M)_{(u, \mu , v)}^{(\eta, \tau, \sigma)}$.
Since $\Lambda_{_Q}$ is a diffeomorphism of $\mathfrak{so}(n)$ which maps $\mathcal{L}_{\mathfrak{so}(n)}(\mathcal{\widehat{J}}_M)_{(u, \mu , v)}^{(\eta, \tau, \sigma)}$ onto 

$\mathcal{L}_{\mathfrak{so}(n)}(M)_{(u, \mu , v)}^{(\eta, \tau, \sigma)}$ , one of these two sets is a submanifold of $\mathfrak{so}(n)$ if and only if the other is too, and in this case they are diffeomorphic.

Now consider the left action $\psi$,  of the compact Lie group $\mathcal{C}_{\mathcal{\widehat{J}}_M} \cap \mathcal{O}(n)$ on $\mathfrak{so}(n)$, defined by 
$\psi_A(X)=AXA^T$, for $A \in \mathcal{C}_{\mathcal{\widehat{J}}_M} \cap \mathcal{O}(n) $ and $X\in\mathfrak{so}(n)$;  
\ $\mathcal{L}_{\mathfrak{so}(n)}(\mathcal{\widehat{J}}_M)_{(u, \mu , v)}^{(\eta, \tau, \sigma)}$ is the orbit of $\widehat{\mathcal{J}}$, while the isotropy subgroup at $\widehat{\mathcal{J}}$ is $\mathcal{C}_{\widehat{{\mathcal{J}}}} \cap \mathcal{O}(n)$; \ so $\mathcal{L}_{\mathfrak{so}(n)}(\mathcal{\widehat{J}}_M)_{(u, \mu , v)}^{(\eta, \tau, \sigma)}$ and  $\mathcal{L}_{\mathfrak{so}(n)}(M)_{(u, \mu , v)}^{(\eta, \tau, \sigma)}$ are  compact submanifolds of $\mathfrak{so}(n)$ both diffeomorphic to $\dfrac{\mathcal{C}_{\mathcal{\widehat{J}}_M} \cap \mathcal{O}(n)}{\mathcal{C}_{\widehat{{\mathcal{J}}}} \cap \mathcal{O}(n)}$ (by Remark \ref{Popov}). This concludes the proof.
\end{proof}

\begin{thm}\label{manifold-O-con-indici}
Let $M \in S \mathcal{O}(n)$ whose real Jordan form, $\mathcal{\widehat{J}}_M$, is as in Notations \ref{notaz_ortog} $(*)$ and
fix any set $(\eta, u, \tau , \mu, \sigma, v)$ of $M$-admissible multi-indices. Then

a) $\mathcal{L}_{\mathfrak{so}(n)}(M)_{(u, \mu , v)}^{(\eta, \tau, \sigma)}$ is a compact homogeneous submanifold of $\mathfrak{so}(n)$, whose connected components are all diffeomorphic to the product
\smallskip

$\Gamma_{(g; u_{_1}, \cdots, u_{_b})} \times \bigg[ \prod\limits_{l=1}^r \Theta_{(\mu_{_{(l,1)}}, \cdots, \mu_{_{(l,d_{_l})}})} \bigg] \times \Gamma_{(0; v_{_1}, \cdots, v_{_c})}$.
\smallskip

b) The manifold $\mathcal{L}_{\mathfrak{so}(n)}(M)_{(u, \mu , v)}^{(\eta, \tau, \sigma)}$ is connected if and only if either $M$ has no real eigenvalues \ \ or\ \ $-1$ is not eigenvalue of $M$ and \ $0$ is eigenvalue of $\mathcal{L}_{\mathfrak{so}(n)}(M)_{(u, \mu , v)}^{(\eta, \tau, \sigma)}$;
\\it has two connected components if and only if either \ $1$ is eigenvalue of $M$, \ $-1$ is not eigenvalue of $M$ and \ $0$ is not eigenvalue of  $\mathcal{L}_{\mathfrak{so}(n)}(M)_{(u, \mu , v)}^{(\eta, \tau, \sigma)}$, \ \ or\ \ $-1$ is eigenvalue of $M$ and \ $0$ is eigenvalue of $\mathcal{L}_{\mathfrak{so}(n)}(M)_{(u, \mu , v)}^{(\eta, \tau, \sigma)}$, \ \ or\ \ $1$ is not eigenvalue of $M$ and \ $-1$ is eigenvalue of $M$; \\it has four connected components if and only if both \ $1$ and $-1$ are eigenvalues of $M$ and \ $0$ is not eigenvalue of $\mathcal{L}_{\mathfrak{so}(n)}(M)_{(u, \mu , v)}^{(\eta, \tau, \sigma)}$. 
\smallskip

Denote by $\mathcal{B}$ an arbitrary connected component of $\mathcal{L}_{\mathfrak{so}(n)}(M)_{(u, \mu , v)}^{(\eta, \tau, \sigma)}$. Then

c) $\mathcal{B}$ is simply connected and $\pi_2 (\mathcal{B})$ is a free abelian group whose rank is \\$b-\delta_{(g,0)}  \delta_{(b,1)} \delta_{(u_{_1},1)} + \ \delta_{(g,2)}(1-\delta_{(b,0)}) \ +\ \sum\limits_{l=1}^r d_{_l} \ -r \ +  \ c\ -\ \delta_{(c,1)}\delta_{(v_{_1},1)} $\ ;

d) assume that all non-real eigenvalues of $\mathcal{L}_{\mathfrak{so}(n)}(M)_{(u, \mu , v)}^{(\eta, \tau, \sigma)}$ are simple and that the multiplicity of \ $0$  as eigenvalue of $\mathcal{L}_{\mathfrak{so}(n)}(M)_{(u, \mu , v)}^{(\eta, \tau, \sigma)}$ is less than or equal to $2$; \ then, \ for every $\alpha \geq 3$,\ \ \ $\pi_{\alpha}(\mathcal{B})$ is isomorphic to the direct sum
\smallskip

$\pi_{\alpha}(S\mathcal{O}(h))\oplus \bigg[ \bigoplus\limits_{l=1}^r  \pi_{\alpha}(U(m_{_{l}})) \bigg] \oplus  \pi_{\alpha}(S\mathcal{O}(2k)) $.

\end{thm}

\begin{proof}
From Lemmas \ref{commutazione} and \ref{Ethetacommutaz}, we get 

$\mathcal{C}_{\mathcal{\widehat{J}}_M} \cap \mathcal{O}(n) =\lbrace GL(h,\mathbb{R}) \oplus \big(\bigoplus_{l=1}^rGL(m_{_l},\mathbb{C})\big)  \oplus GL(2k,\mathbb{R}) \rbrace \cap \mathcal{O}(n) = \\ 
\mathcal{O}(h) \oplus \big(\bigoplus_{l=1}^r U(m_{_l}) \big) \oplus \mathcal{O}(2k)\ ,
$ 

$\mathcal{C}_{\widehat{\mathcal{J}}} \cap \mathcal{O}(n)=$

$\big\lbrace GL(g,\mathbb{R}) \oplus \big(\bigoplus\limits_{i=1}^b GL(u_{_i}, \mathbb{C}\big) \oplus \big( \bigoplus\limits_{l=1}^r \ \bigoplus\limits_{t=1}^{d_{_l}}G(\mu_{_{(l,t)}},\mathbb{C})\big)  \oplus \big(\bigoplus\limits_{j=1}^c GL(v_{_j},\mathbb{C})\big) \big\rbrace \cap \mathcal{O}(n)  =
\mathcal{O}(g) \oplus \big(\bigoplus\limits_{i=1}^b U(u_{_i})\big) \oplus \big(\bigoplus\limits_{l=1}^r \ \bigoplus\limits_{t=1}^{d_{_l}}U(\mu_{_{(l,t)}})\big) \oplus (\bigoplus\limits_{j=1}^c U(v_{_j}))$.
\smallskip

It is easy to prove that the quotient $\dfrac{\mathcal{C}_{\mathcal{\widehat{J}}_M} \cap \mathcal{O}(n)}{\mathcal{C}_{\widehat{\mathcal{J}}} \cap \mathcal{O}(n)}$ is naturally diffeomorphic to 
to the product 
$$\bigg[ \ \dfrac{\mathcal{O}(h)}{\mathcal{O}(g) \oplus \big(\bigoplus\limits_{i=1}^b U(u_{_i})\big)} \ \bigg] \times \bigg[ \ \prod\limits_{l=1}^r \dfrac{U(m_{_l})}{(\bigoplus\limits_{t=1}^{d_{_l}}U(\mu_{_{(l,t)}})\big)} \ \bigg] \times \bigg[ \ \dfrac{\mathcal{O}(2k)}{\big(\bigoplus\limits_{j=1}^c U(v_{_j})\big)} \ \bigg]  \ \ \ \ .$$
(In this formula we have assumed, without losing generality, that $h, k \geq 1$.) Then we get (a), taking into account Proposition \ref{log-ortog} and Remarks-Definitions \ref{def-spaz-omog} (d).

The statement (b) still follows, by means of simple considerations, from Remarks-Definitions \ref{def-spaz-omog} (d).

Part (c) easily follows from Remark \ref{Pi2-con delta}.

The statement in (d) follows from Propositions \ref{Gamma-omotopia} (c) and \ref{Theta-omotopia} (c), taking into acount that, in this case, we have $g \leq 2, \ \ g+2b=h,\ \ c=k$ \ and \ $m_{_l}=d_{_l}$, \ for every \ $l=1, \cdots, r \ .$
\end{proof}

\begin{rems}\label{log-numer}
a) Remembering Notations \ref{notaz_ortog} and Definitions \ref{skew_symmetric_logarithms} (a), we remark that, if the set of multi-indices 
$(\eta, u, \tau , \mu, \sigma, v)$ is $M$-admissible and the order $n$ of $M$ is odd, then $0$ is necessarily an eigenvalue of $\mathcal{L}_{\mathfrak{so}(n)}(M)_{(u, \mu , v)}^{(\eta, \tau, \sigma)}$ 
.

b) For $n \geq 2$, the set of all skew-symmetric real logarithms of the $n \times n$ special orthogonal matrix $M$ is never finite and it is countably infinite if and only if, for every 
$M$-admissible set of multi-indices 
$(\eta, u, \tau , \mu, \sigma, v)$, the manifold $\mathcal{L}_{\mathfrak{so}(n)}(M)_{(u, \mu , v)}^{(\eta, \tau, \sigma)}$ has zero dimension.
From Theorem \ref{manifold-O-con-indici} and Remarks-Definitions \ref{def-spaz-omog} (b),(c), we easily obtain that, if $M \in S\mathcal{O}(n)$, then
the set $\mathcal{L}og_{{\mathfrak{so}(n)}}(M) $ is countably infinite if and only if all non-real eigenvalues of $M$ are simple and the multiplicity of $1$ and $-1$ as (possible) eigenvalues of $M$ is less than or equal to $2$.
\end{rems}

\begin{defi}\label{prin-antis-log}
Let $M \in S\mathcal{O}(n)$. As in Definition \ref{prin-gen}, we say that a matrix $X \in \mathfrak{so}(n)$ is a \emph{principal skew-symmetric real logarithm} of $M$, if $exp(X)=M$ and each eigenvalue of $X$ has imaginary part in $[-\pi, \pi]$.
We denote by $\mathcal{PL}og_{{\mathfrak{so}(n)}}(M)$ the set of all principal skew-symmetric real logarithms of $M$.
\end{defi}

\begin{thm}\label{antis-log-princ} 
Let $M \in S\mathcal{O}(n)$. If \ $-1$ is not eigenvalue of $M$, then the set \\$\mathcal{PL}og_{{\mathfrak{so}(n)}}(M)$ consists of a single point, while, if \ $-1$ is an eigenvalue of $M$ of multiplicity $2k \geq 2$, then $\mathcal{PL}og_{{\mathfrak{so}(n)}}(M)$ is a compact submanifold of ${\mathfrak{so}(n)}$, diffeomorphic to the homogeneous space $\dfrac{\mathcal{O}(2k)}{ U(k)}$. In this last case, $\mathcal{PL}og_{{\mathfrak{so}(n)}}(M)$ has two  connected components, each of which is diffeomorphic to the symmetric space $\Gamma_{(0 ; k)}$, and hence, called \ $\mathcal{C}$\  any of its two connected components, we have that \\
a) \ $\mathcal{C}$ is a single point, if $k=1$;\\
b) \ $\mathcal{C}$ is diffeomorphic to a $2$-dimensional sphere, if $k=2$;\\
c) \ $\mathcal{C}$ is simply connected and $\pi_2(\mathcal{C})$ is an infinite cyclic group, if $k \geq 3$.
\end{thm}

\begin{proof}
Since the set \ $\mathcal{PL}og_{{\mathfrak{so}(n)}}(M)$ agrees with the manifold $\mathcal{L}_{\mathfrak{so}(n)}(M)_{(u, \mu , v)}^{(\eta, \tau, \sigma)}$, where
$\eta = \mathsf{O}$, $\tau = \mathsf{O}$, $\sigma= \mathsf{O}$,  $\mu=(m_{_1};\ \cdots\ ; m_{_r})$, $u=(h)$, $v=(k)$, the statement follows from Theorem \ref{manifold-O-con-indici} (and from its proof), taking into account Remarks-Definitions \ref{def-spaz-omog} (b),(c),(d) and Proposition \ref{Bott}.
\end{proof}

\end{document}